\documentclass[12pt, twoside]{article}
\usepackage{amsmath,amsthm,amssymb}
\usepackage{times}
\usepackage{enumerate}

\usepackage{hyperref}
\usepackage{amsfonts}
\usepackage[utf8]{inputenc}

\def\titlerunning#1{\gdef\titrun{#1}}
\makeatletter
\def\author#1{\gdef\autrun{\def\and{\unskip, }#1}\gdef\@author{#1}}
\def\address#1{{\def\and{\\\hspace*{18pt}}\renewcommand{\thefootnote}{}%
\footnote {#1}}%
\markboth{\autrun}{\titrun}}
\makeatother
\def\email#1{e-mail: #1}
\def\subjclass#1{{\renewcommand{\thefootnote}{}%
\footnote{\emph{Mathematics Subject Classification (2010):} #1}}}
\def\keywords#1{\par\medskip
\noindent\textbf{Keywords.} #1}

\newtheorem{thm}{Theorem}[section]

\newtheorem{lem}[thm]{Lemma}

\theoremstyle{definition}
\newtheorem{dfn}[thm]{Definition}
\newtheorem{rem}[thm]{Remark}

\numberwithin{equation}{section}

\newcommand{\ap}[1]{\left\langle#1\right\rangle}

\def\R{\mathbb{R}}
\def\rr{\mathbb{R}}
\def\dR{\mathbb{R}}

\def\bZ{\mathbf{Z}}
\def\bb{\mathbf{b}}
\def\bB{\mathbf{B}}

\DeclareMathOperator{\law}{law}

\newcommand{\X}{{\overline{X}}}
\newcommand{\V}{{\overline{V}}}
\newcommand{\Z}{{\overline{Z}}}
\newcommand{\E}{{\mathbb{E}}}
\newcommand{\ee}{{\mathbb{E}}}
\renewcommand{\P}{{\mathcal{P}}}

\usepackage{bbold}
\newcommand{\1}{{\mathbb{1}}}

\begin{document}


\baselineskip=17pt


\titlerunning{}

\title{Stochastic Mean-Field Limit: \\ Non-Lipschitz Forces \& Swarming}

\author{Fran\c cois~Bolley \and Jos\'e~A.~Ca\~nizo \and Jos\'e~A.~Carrillo}

\date{September 24, 2010}

\maketitle

\address{F. Bolley: Ceremade, UMR CNRS 7534, Universit\'e
Paris-Dauphine, Place du Mar\'echal De Lattre De Tassigny, F-75016
Paris, France; \email{bolley@ceremade.dauphine.fr} \and J.~A.
Ca\~nizo: Departament de Matem\`atiques, Universitat Aut\`onoma de
Barcelona, E-08193 Bellaterra, Spain; \email{canizo@mat.uab.es}
\and J.~A. Carrillo: Instituci\'o Catalana de Recerca i Estudis
Avan\c cats and Departament de Matem\`atiques, Universitat
Aut\`onoma de Barcelona, E-08193 Bellaterra, Spain;
\email{carrillo@mat.uab.es}}



\begin{abstract}
  We consider general stochastic systems of interacting particles with
  noise which are relevant as models for the collective behavior of
  animals, and rigorously prove that in the mean-field limit the
  system is close to the solution of a kinetic PDE. Our aim is to
  include models widely studied in the literature such as the
  Cucker-Smale model, adding noise to the behavior of individuals. The
  difficulty, as compared to the classical case of globally Lipschitz
  potentials, is that in several models the interaction potential between particles is only
  locally Lipschitz, the local Lipschitz constant growing to infinity
  with the size of the region considered. With this in mind, we
  present an extension of the classical theory for globally Lipschitz
  interactions, which works for only locally Lipschitz ones.
\end{abstract}

\section{Introduction}

The formation of large-scale structures (patterns) without the need of
leadership (self-organization) is one of the most interesting and not
completely understood aspect in the collective behavior of certain
animals, such as birds, fish or insects.  This phenomena has attracted
lots of attention in the scientific community, see
\cite{camazine,couzin,parrish,yates-etal} and the references therein.

Most of the proposed models in the literature are based on
particle-like description of a set of large individuals; these
models are called Individual-Based Models (IBM). IBMs typically
include several interactions between individuals depending on the
species, the precise mechanism of interaction of the animals and
their particular biological environment. However, most of these
IBMs include at least three basic effects: a short-range
repulsion, a long-range attraction and a ``mimicking'' behavior
for individuals encountered in certain spatial regions. This
so-called three-zone model was first used for describing fish
schools in \cite{Ao,HW} becoming a cornerstone of swarming
modelling, see \cite{BTTYB,HH}.

The behavior of a large system of individuals can be studied through
mesoscopic descriptions of the system based on the evolution of the
probability density of finding individuals in phase space. These
descriptions are usually expressed in terms of space-inhomogeneous kinetic
PDEs and the scaling limit of the interacting particle system to
analyze is usually called the \emph{mean-field limit}. These kinetic
equations are useful in bridging the gap between a microscopic
description in terms of IBMs and macroscopic or hydrodynamic
descriptions for the particle probability density. We refer to the
review \cite{cftv10} for the different connections between these
models and for a larger set of references.

The mean-field limit of deterministic interacting particle systems
is a classical question in kinetic theory, and was treated in
\cite{BH,dobru,neunzert2} in the case of the Vlasov equation. In
these papers, the particle pairwise interaction is given by a
globally bounded Lipschitz force field. Some of the recent models
of swarming introduced in \cite{DCBC,CS2,HT08} do not belong to
this class due to their growth at infinity leading to an
interaction kernel which is only locally Lipschitz. These IBMs are
kinetic models in essence since the interactions between
individuals are at the level of the velocity variable to ``align''
their movements for instance or to impose a limiting ``cruising
speed''. The mean-field limit for deterministic particle systems
for some models of collective behavior with locally Lipschitz
interactions was recently analysed in \cite{ccr} showing that they
follow the expected Vlasov-like kinetic equations.

On the other hand, noise at the level of the IBMs is an important
issue since we cannot expect animals to react in a completely
deterministic way. Therefore, including noise in these IBMs and thus,
at the level of the kinetic equation is an important modelling
ingredient. This stochastic mean-field limit formally leads to kinetic
Fokker-Planck like equations for second order models as already
pointed out in \cite{CDP}. The rigorous proof of this stochastic
mean-field limit has been carried out for globally Lipschitz
interactions in \cite{Szn91,Mel96}, see also~\cite{mckean}.

This work is devoted to the rigorous analysis of the stochastic
mean-field limit of locally Lipschitz interactions that include
relevant swarming models in the literature such as those in
\cite{DCBC,CS2}. We will be concerned with searching the rate of
convergence, as the number of particles $N\to\infty$, of the
distribution of each of the particles and of the empirical measure
of the particle system to the solution of the kinetic equation.
This convergence  will also establish the propagation of chaos as
$N\to\infty$ for the particle system and will be measured in terms
of distances between probability measures. Here, we will not deal
with uniform in time estimates since no stabilizing behaviour can
be expected in this generality, such estimates were obtained only
in a specific instance of Vlasov-Fokker-Planck equation,
see~\cite{BGM10}. The main price to pay to include possible growth
at infinity of the Lipschitz constants of the interaction fields
will be at the level of moment control estimates. Then, there will
be a trade-off between the requirements on the interaction and the
decay at infinity of the laws of the processes at the initial
time.

The work is organized as follows: in the next two subsections we will
make a precise descriptions of the main results of this work, given in
Theorems \ref{thm:main} and \ref{thm:mainexist} below, together with a
small overview of preliminary classical well-known facts and a list of
examples, variants and particular cases of applications in swarming
models. The second section includes the proof of the stochastic
mean-field limit of locally Lipschitz interacting particle systems
under certain moment control assumptions (thus proving Theorem
\ref{thm:main}). Finally, the third section will be devoted to the
proof of Theorem \ref{thm:mainexist}: a result of existence and
uniqueness of the nonlinear partial differential equation and its
associated nonlinear stochastic differential equation, for which the
stochastic mean-field limit result can be applied. The argument will
be performed in the natural space of probability measures by an
extension to our diffusion setting of classical characteristics
arguments for transport equations.

\subsection{Main results} \label{sec:main}

We will start by introducing the two instances of IBMs that
triggered this research. The IBM proposed in \cite{DCBC} includes
an effective pairwise potential $U:\R^d\longrightarrow \R$
modeling the short-range repulsion and long-range attraction. The
only ``mimicking'' interaction in this model is encoded in a
relaxation term for the velocity arising as the equilibrium speed
from the competing effects of self-propulsion and friction of the
individuals. We will refer to it as the D'Orsogna et al model in
the rest. More precisely, this IBM for $N$-particles in the
mean-field limit scaling reads as:
\begin{equation*}
  \left\lbrace
    \begin{array}{ll}
      \displaystyle \frac{dX^i}{dt} = V^i,
      &
      \vspace{.3cm}
      \\
      \displaystyle \frac{dV^i}{dt} = (\alpha - \beta \,|V^i|^2) V^i
      - \frac1N \sum_{j \neq i } \nabla U (|X^i - X^j|),
      &
    \end{array}
  \right.
\end{equation*}
where $\alpha>0$ measures the self-propulsion strength of individuals,
whereas the term corresponding to $\beta>0$ is the friction assumed to
follow Rayleigh's law. A typical choice for $U$ is a smooth radial
potential given by
$$
U(x) = -C_A e^{-|x|^2 /\ell_A^2} + C_R e^{-|x|^2 /\ell_R^2}.
$$
where $C_A, C_R$ and $\ell_A, \ell_R$ are the strengths and the
typical lengths of attraction and repulsion, respectively.

The other motivating example introduced in \cite{CS2} only
includes an ``alignment'' or reorientation interaction effect and we
will refer to it as the Cucker-Smale model. Each individual in the
group adjust their relative velocity by averaging with all the
others. This averaging is weighted in such a way that closer
individuals have more influence than further ones. For a system
with $N$ individuals the Cucker-Smale model in the mean-field
scaling reads as
$$
\left\{\begin{array}{lr} \displaystyle\frac{dX^i}{dt} = V^i, \\[3mm]
\displaystyle\frac{dV_i}{dt} = \frac1N \displaystyle\sum_{j=1}^{N}
w_{ij} \left(V^j - V^i\right) ,
\end{array}\right.
$$
with the \emph{communication rate} matrix given by:
$$
w_{ij} = w(|X^i-X^j|)= \frac 1 {\left(1 +|X^i - X^j|^2
  \right)^\gamma}
$$
for some $\gamma \geq 0$. We refer to \cite{CS2,HT08,cfrt09,cftv10}
and references therein for further discussion about this model and
qualitative properties. Let us remark that both can be considered
particular instances of a general IBM of the form
\begin{equation}
\label{eq:odesys}
\begin{cases}
  \displaystyle\frac{dX^i}{dt} = V^i \\
\displaystyle \frac{dV^i}{dt} = - F(X^i, V^i) - \frac{1}{N} \sum_{j=1}^N H(X^i - X^j, V^i-V^j) dt, \qquad 1 \leq i \leq N\\
\end{cases}
\end{equation}
where $F,H:\R^{2d}\longrightarrow \R$ are suitable functions: the
D'Orsogna et al model with $F(x,v)=(\beta \vert v \vert^2 -
\alpha) v$ and $H(x,v) = \nabla_x U(x)$ and the Cucker-Smale model
with $F=0$ and $H(x,v)=w(x)v$. Let us emphasize that $F$ in the
D'Orsogna et al model and $H$ in the Cucker-Smale model are not
globally Lipschitz functions in $\R^{2d}$.

Our aim is to deal with a general system of interacting particles
of the type \eqref{eq:odesys} with added noise and suitable
hypotheses on $F$ and $H$ including our motivating examples. More
precisely, we will work then with a general large system of $N$
interacting $\rr^{2d}$-valued processes $(X^{i}_t, V^{i}_t)_{t\geq
0}$ with $1 \leq i \leq N$ solution of
\begin{equation}
\label{eq:sdesys}
\begin{cases}
  d X_t^i = V_t^i dt, \\
\displaystyle{dV_t^i = \sqrt{2} dB_t^i - F(X_t^i, V_t^i) dt - \frac{1}{N} \sum_{j=1}^N H(X_t^i - X_t^j, V_t^i-V_t^j) dt,} \\
\end{cases}
\end{equation}
with independent and commonly distributed initial data $(X^{i}_0,
V^{i}_0)$ with $1 \leq i \leq N$. Here, and throughout this paper, the
$(B^i_t)_{t \geq 0}$ are $N$ independent standard Brownian motions in
$\rr^d$. More general diffusion coefficients will be considered in the
next subsection. The asymptotic behavior of the Cucker-Smale system
with added noise has been recently considered in \cite{CM}, and
eq. (\ref{eq:sdesys}) includes as a particular case the
continuous-time models discussed there. Our main objective will be to
study the large-particle number limit in their mean-field limit
scaling. It is sometimes usual to write $(X^{i,N}_t,V^{i,N}_t)$ to
track $N$ individuals, but to avoid a cumbersome notation we will drop
the superscript $N$ unless the dependence on it needs to be
emphasized.

By symmetry of the initial configuration and of the evolution, all
particles have the same distribution on $\rr^{2d}$ at time $t$,
which will be denoted $f_t^{(1)}$. For any given $t>0$ the
particles get correlated due to the nonlocal term
$$
- \frac{1}{N} \sum_{j=1}^N H(X_t^i - X_t^j,
V_t^i-V_t^j)
$$
in the evolution, though they are independent at initial time.
But, since the pairwise action
of two particles $i$ and $j$ is of order $1/N$,
it seems reasonable that two of these interacting particles (or a
fixed number $k$ of them) become less and less correlated as $N$
gets large: this is what is called propagation of chaos. The
statistical quantities of the system are given by the empirical
measure
$$
\hat{f}_t^N = \frac{1}{N} \sum_{i=1}^N \delta_{(X^{i}_t,V^{i}_t)}.
$$
It is a general fact, see Sznitman \cite{Szn91}, that propagation of
chaos for a symmetric system of interacting particles is equivalent to
the convergence in $N$ of their empirical measure.  Following
\cite{Szn91} we shall prove quantitative versions of these equivalent
results.

We shall show that our $N$ interacting processes
$({X}^i_t,{V}^i_t)_{t \geq 0}$ respectively behave as $N\to\infty$
like the processes $(\X^i_t,\V^i_t)_{t \geq 0}$, solutions of the
kinetic McKean-Vlasov type equation on $\dR^{2d}$
\begin{equation}
  \label{eq:nlSDE}
  \left\{
    \begin{split}
      &d\X_t^i = \V_t^i \,dt
      \\
      &d\V_t^i = \sqrt{2}\, dB_t^i- F(\X_t^i, \V_t^i)dt - H \ast f_t
      (\X_t^i,\V_t^i)dt,
      \\
      &(\X_0^i,\V_0^i)=(X_0^i,V_0^i), \quad f_t = \law(\X_t^i, \V_t^i).
    \end{split}
  \right.
\end{equation}
Here the Brownian motions $(B^i_t)_{t \geq 0}$ are those governing the
evolution of the $(X^i_t,V^i_t)_{t \geq 0}$. Note that the above set
of equations involves the condition that $f_t$ is the distribution of
$(\X^i_t,\V^i_t)$, thus making it nonlinear. The processes
${(\X^i_t,\V^i_t)_{t \geq 0}}$ with $i\geq 1$ are independent since
the initial conditions and driving Brownian motions are
independent. Moreover they are identically distributed and, by the
It\^o formula, their common law $f_t$ at time $t$ should evolve
according to
\begin{equation}\label{eq:pde}
  \partial_t f_t + v  \cdot \nabla_x f_t  = \Delta_v f_t + \nabla_v \cdot ((F + H*f_t) f_t), \quad t>0, x,v \in \rr^d.
\end{equation}
Here $a \cdot b$ denotes the scalar product of two vectors $a$ and $b$
in $\rr^{d}$ and $*$ stands for the convolution with respect to $(x,v)
\in \rr^{2d}$:
$$
H* f (x) = \int_{\rr^{2d}} H(x-y, v-w) \, f(y,w) \, dy \, dw.
$$
Moreover, $\nabla_x$ stands for the gradient with respect to the
position variable $x \in \rr^d$ whereas $\nabla_v$, $\nabla_v \cdot$
and $\Delta_v$ respectively stand for the gradient, divergence and
Laplace operators with respect to the velocity variable $v \in
\rr^{d}$.

Assuming the well-posedness of the stochastic differential
system \eqref{eq:sdesys} and of the nonlinear equation \eqref{eq:nlSDE}
together with some uniform moment bounds, we will
obtain our main result on the stochastic mean-field limit. Existence
and uniqueness of solutions to~\eqref{eq:sdesys},~\eqref{eq:nlSDE}
and~\eqref{eq:pde} verifying the assumptions of the theorem will also
be studied but with more restrictive assumptions on $F$ and $H$ that
we will comment on below.

\begin{thm}\label{thm:main}
  Let $f_0$ be a Borel probability measure and $(X^i_0, V^i_0)$ for $1
  \leq i \leq N$ be $N$ independent variables with law $f_0$.  Assume
  that the drift $F$ and the antisymmetric kernel $H$, with
  $H(-x,-v)=-H(x,v)$, satisfy that there exist constants
  $A, L, p > 0$ such that
  \begin{align}
    \label{hypaf1}
    -(v - w) \cdot (F(x,v) - F(x,w)) &\leq A \,\vert v- w \vert^2
    \\
    \label{hypaf2}
    \vert F(x,v) - F(y,v) \vert &\leq L \min\{\vert x-y \vert, 1 \}
    (1 + \vert v \vert^p)
  \end{align}
  for all $x,y,v,w$ in $\rr^d,$ and analogously for $H$ instead of
  $F$. Take $T > 0$. Furthermore, assume that the particle
  system~\eqref{eq:sdesys} and the processes~\eqref{eq:nlSDE} have
  global 
  solutions on $[0,T]$ with initial data
  $(X^i_0, V^i_0)$ such that
  \begin{equation}\label{hyp:unifmoment}
    \sup_{0 \leq t \leq T} \!\!\Big\{ \!\int_{\rr^{4d}} \!\!\!\vert H(x-y, v-w) \vert^2 df_t(x,v)  df_t(y,w)
    + \!  \int_{\rr^{2d}} \!\!(|x|^2 + e^{a|v|^p}) df_t(x,v) \Big\} < +\infty,
  \end{equation}
  with $f_t=law(\X^i_t,\V^i_t)$. Then there exists a constant $C>0$
  such that
  \begin{equation}\label{eq:main1}
    \E\big[|X_t^i - \X_t^i|^2 + |V_t^i - \V_t^i|^2\big]
    \leq
    \frac{C}{N^{e^{-Ct}}}
  \end{equation}
  for all $0 \leq t \leq T$ and $N \geq 1$.

  Moreover, if additionally there exists $p'>p$ such that
  \begin{equation}\label{hyp:unifmoment2}
    \sup_{0 \leq t \leq T}   \int_{\rr^{2d}} e^{a|v|^{p'}} df_t(x,v) < +\infty,
  \end{equation}
  then for all $0 < \epsilon < 1$ there exists a constant $C$ such
  that
  \begin{equation}\label{eq:main2}
    \E\big[|X_t^i - \X_t^i|^2 + |V_t^i - \V_t^i|^2\big]
    \leq
    \frac{C}{N^{1-\epsilon}}
  \end{equation}
  for all $0 \leq t \leq T$ and $N \geq 1$.
\end{thm}

This result classically ensures quantitative estimates on the mean
field limit and the propagation of chaos. First of all, it ensures
that the common law $f_t^{(1)}$ of any (by exchangeability) of the
particles $X^i_t$ at time $t$ converges to $f_t$ as $N$ goes to
infinity, as we have
\begin{equation}
  \label{prelim1}
  W_2^2 (f^{(1)}_t, f_t ) \leq \ee \big[ \vert X^i_t - \X^i_t
  \vert^2 + \vert V^i_t - \V^i_t \vert^2 \big] \leq \varepsilon(N)
\end{equation}
Here $W_2$ stands for the Wasserstein distance between two
measures $\mu$ and $\nu$ in the set $\mathcal P_2(\rr^{2d})$ of
Borel probability measures on $\mathbb R^{2d}$ with finite moment
of order $2$ defined by
$$
W_2 (\mu, \nu ) = \inf_{(Z,\Z)} \left\{ \ee\left[ \vert Z - \Z
    \vert^2 \right]\right\}^{1/2},
$$
where the infimum runs over all couples of random variables $(Z,\Z)$ in $\rr^{2d} \times \rr^{2d}$
with $Z$ having law $\mu$ and $\Z$ having law $\nu$ (see
\cite{villani-stflour} for instance). Moreover $\varepsilon(N)$ denotes the quantity in the right hand side of
\eqref{eq:main1} or \eqref{eq:main2}, depending on which part of
Theorem \ref{thm:main} we are using.

Moreover, it proves a quantitative version of propagation of chaos:
for all fixed $k$, the law $f^{(k)}_t$ of any (by exchangeability) $k$
particles $(X^i_t, V^i_t)$ converges to the tensor product
$f_t^{\otimes k}$ as $N$ goes to infinity, according to
\begin{align*}
W_2^2 (f^{(k)}_t , f_t^{\otimes k})
 \leq&\,
 \ee \left[\big\vert (X^1_t, V^1_t, \cdots, X^k_t, V^k_t) - (\X^1_t, \V^1_t, \cdots, \X^k_t, \V^k_t)
 \big\vert^2\right]\\
 \leq&\,
 k \,\ee \left[ \vert X^1_t - \X^1_t \vert^2 + \vert V^1_t - \V^1_t \vert^2 \right]
 \leq
 k \varepsilon(N).
\end{align*}

It finally gives the following quantitative result on the
convergence of the empirical measure $\hat{f}^N_t$ of the particle
system to the distribution $f_t$ : if $\varphi$ is a Lipschitz map
on $\rr^{2d}$, then
\begin{align*}
\ee &\left[\Big\vert \frac{1}{N} \sum_{i=1}^N  \varphi (X^i_t,
V^i_t) - \int_{\rr^{2d}} \varphi \, df_t \Big\vert^2\right] \\
&\leq 2 \, \ee \left[\vert \varphi(X^i_t, V^i_t) - \varphi(\X_t^i,
\V_t^i) \vert^2 +  \,\Big\vert \frac{1}{N} \sum_{i=1}^N
\varphi(\X^i_t, \V_t^i) - \int_{\rr^{2d}} \varphi \, df_t
\Big\vert^2 \right] \leq \varepsilon(N) + \frac{C}{N}
\end{align*}
by Theorem~\ref{thm:main} and argument on the independent variables
$(\X^i_t, \V^i_t)$ based on the law of large numbers,
see~\cite{Szn91}.

The argument of Theorem \ref{thm:main} is classical for globally
Lipschitz drifts \cite{Szn91,Mel96}. For space-homogeneous kinetic
models it was extended to non-Lipschitz drifts by means of convexity
arguments, first in one dimension in~\cite{BRTV98}, then more
generally in any dimension in~\cite{CGM08,Mal01}. Here, in our space
inhomogeneous setting, sole convexity arguments are hopeless, and we
will replace them by moment arguments using hypothesis
\eqref{hyp:unifmoment}. We also refer to \cite{mco05,BuCM} for related
problems and biological discussions in space-homogeneous kinetic
models with globally Lipschitz drifts but nonlinear diffusions.

Our proof will be written for $p>0$, but one can simplify it with
$p=0$, by only assuming finite moments of order $2$ in position and
velocity. In this case our proof is the classical Sznitman's
proof for existence, uniqueness, and mean-field limit for globally
Lipschitz drifts, written in our kinetic setting and giving the classical
decay rate in \eqref{eq:main1} as $1/N$, compared to
\eqref{eq:main1}-\eqref{eq:main2}. We will discuss further examples
related to swarming models and extensions in subsection~\ref{sec:ex}.

Section~\ref{sec:existence} will be devoted to the proof of
existence, uniqueness, and moment propagation properties
\eqref{hyp:unifmoment} and \eqref{hyp:unifmoment2} for solutions
to~\eqref{eq:sdesys},~\eqref{eq:nlSDE} and~\eqref{eq:pde}. This
well-posedness results and moment control for solutions will be
obtained under more restrictive assumptions that those used in the
proof of Theorem \ref{thm:main}.

\begin{thm}\label{thm:mainexist}
  Assume that the drift $F$ and the kernel $H$ are locally Lipschitz
  functions satisfying that there exist $C,L \geq 0$ and $0 < p \leq
  2$ such that
  \begin{align}
    -v \cdot F(x,v) &\leq C (1 + |v|^2)
    \label{eq:F-one-sided-growth} \\
    - (v - w) \cdot (F(x,v) - F(x,w))
    &\leq
    L | v- w |^2 (1 + | v |^p + \vert w \vert^p),
    \label{eq:F-loc-Lipschitz-v}\\
    | F(x,v) - F(y,v) |
    &\leq
    L | x-y | (1 + | v |^p ),
    \label{eq:F-loc-Lipschitz-x}\\
    |H(x,v)| &\leq C (1 + |v|),
    \label{eq:H-two-sided-sys}\\
    |H(x,v) - H(y,w)| &\leq
    L (|x-y| + |v-w|)(1+|v|^p + |w|^p),
    \label{eq:H-loc-Lipschitz}
  \end{align}
  for all $x,v,y,w \in \R^d$. Let $f_0$ be a Borel probability measure
  on $\rr^{2d}$ such that
  $$
  \int_{\rr^{2d}} \big( \vert x \vert^2 + e^{a \vert v \vert^{p'}}
  \big) \, df_0(x,v) \, < + \infty.
  $$
  for some $p' \geq p$. Finally, let $(X^i_0, V^i_0)$ for $1 \leq i
  \leq N$ be $N$ independent variables with law $f_0$. Then,
  \begin{enumerate}
  \item[i)] There exists a pathwise unique global solution to the
    SDE~\eqref{eq:sdesys} with initial data $(X^i_0, V^i_0)$.
  \item[ii)] There exists a pathwise unique global solution to the
    nonlinear SDE~\eqref{eq:nlSDE} with initial datum $(X^i_0, V^i_0)$.
  \item[iii)] There exists a unique global solution to the nonlinear PDE
    \eqref{eq:pde} with initial datum$f_0$.
  \end{enumerate}
  Moreover, for all $T>0$ there exists $b>0$ such that
  $$
  \sup_{0 \leq t \leq T} \int_{\rr^{2d}} \big( \vert x \vert^2 +
  e^{b \vert v \vert^{p'}} \big) \, df_t(x,v) < + \infty.
  $$
\end{thm}

\

Concerning the hypotheses on $F$, let us remark that we could also
ask $F$ to satisfy similar properties as $H$ in
\eqref{eq:H-two-sided-sys}--\eqref{eq:H-loc-Lipschitz}, but
\eqref{eq:F-one-sided-growth}--\eqref{eq:F-loc-Lipschitz-x} are
slightly weaker.


\subsection{Examples, extensions and variants}\label{sec:ex}

As discussed above the drift $F$ models exterior or local effects,
such as self propulsion, friction and confinement. In our motivating
examples $F(x,v) = 0$ in the Cucker-Smale model and $F(x,v) =
(\beta \vert v \vert^2 - \alpha)v$ in the D'Orsogna et al model.
On the other hand, $H$ models the interaction between individuals
at $(x,v)$ and $(y,w)$ in the phase space being $H(x,v) = a(x) v $ with $a(x) = (1 + \vert x
\vert^2)^{-\gamma}$, $\gamma >0$
in the Cucker-Smale model  and $H(x,v)= - \nabla U(x)$
in the D'Orsogna et al model. It is straightforward to check the
assumptions of Theorems \ref{thm:main} and \ref{thm:mainexist} in
these two cases.

Of course, more general relaxation terms towards fixed ``cruising
speed'' are allowed in the assumptions of Theorem \ref{thm:main},
for instance: $F(x,v) = (\beta(x) \vert v \vert^\delta -
\alpha(x))v$ with $\alpha, \beta$ globally Lipschitz bounded away
from zero and infinity functions and $\delta>0$. Also, concerning
the interaction kernel we may allow $H(x,v)=a(x) \vert v
\vert^{q-2} v$ with $q\geq 1$ for a bounded and Lipschitz $a$ in
Cucker Smale as introduced in \cite{HaHaKim}. This has the effect
of changing the equilibration rate towards flocking, see
\cite{cfrt09,HaHaKim} for details. However, the assumptions on
existence and moment control in Theorem \ref{thm:mainexist} are
only verified for $q=2$. Other more general mechanisms can be
included such as the one described in \cite{LLE}.

\subsubsection{Variants on the assumptions}

We first remark two simple extensions of the results in
Theorem~\ref{thm:main} by trading off growth control on $F$ and $H$
by moment control of the solutions to \eqref{eq:sdesys}:
\begin{itemize}
\item[{\bf V1.}] Theorem~\ref{thm:main} holds while weakening
assumption \eqref{hypaf1} on $F$ and $H$ to
$$
(v-w) \cdot (F(x,v) - F(x,w)) \geq - A \vert v - w \vert^2 (1 +
\vert v \vert^p + \vert w \vert^p)
$$
both for $F$ and $H$. Solutions in this case need to satisfy
$$
\sup_{N \geq 1} \sup_{0 \leq t \leq T} \E\left[ e^{b \vert V^1_t
\vert^p}\right] <+ \infty
$$
on the particle system or equivalent conditions on $p'$ for the
second estimate \eqref{eq:main2}. Observe by lower continuity and
weak convergence in $N$ of the law of $(X^1_t, V^1_t)$ to the law
$f_t$ of $(\X^1_t, \V^1_t)$ that this is a stronger assumption
than part of the assumption \eqref{hyp:unifmoment} made in
Theorem~\ref{thm:main}, more precisely
$$
\sup_{0 \leq t \leq T} \E\left[ e^{b \vert \V^1_t \vert^p}\right]
< +\infty\, .
$$
Observe also that we may not have global existence in this case
since for instance $F(x,v) = - v^3$ on $\rr^2$, which leads to
blow up in finite time, satisfies this new condition with $p=2$.

\item[{\bf V2.}] One can remove the antisymmetry assumption on $H$ in Theorem~\ref{thm:main} by
  imposing
  \begin{equation*}
    | H(x,v) - H(x,w)| \leq A \,\vert v- w \vert
  \end{equation*}
instead of \eqref{hypaf1} for $H$. The reader can check that very
little modifications are needed at the only point in the proof below
where the symmetry of $H$ is used, namely, when bounding term
$I_{21}$. Actually, one can directly carry out the estimates
instead of symmetrizing the term first. From the modeling point of
view, it is important to include the non-antisymmetric case since
some more refined swarming IBMs include the so-called ``cone of
vision'' or ``interaction region''. In these models, individuals
cannot interact with all the others but rather to a restricted set
of individuals they actually see or feel, see \cite{LLE,cftv10,AIR10}.
From the mathematical point of view this implies that the
interaction term $H \ast f_t$ need not always be a convolution but must be replaced by$$
  H [f_t] (x,v) = \int_{\rr^{2d}} H(x,v; y,w) \, df_t(y,w);
$$
here $H(x,v; \cdot , \cdot )$ is compactly supported in a region that
depends on the value of $(x,v)$ and $H(y,w; x,v)$ is not necessarily equal to $- H(x,v; y,w)$.
Our results extend to this case.

\item[{\bf V3.}] Theorem~\ref{thm:main} also holds when $F$ is an
exterior drift in position only, non globally Lipschitz, for
instance satisfying
$$
\vert F(x) - F(y) \vert
\leq
A \vert x - y \vert (1 + \vert x \vert^q + \vert y \vert^q)
$$
with $q> 0$. Now, the moment control condition
\eqref{hyp:unifmoment} has to be reinforced by assuming
$$
\sup_{0 \leq t \leq T} \E\left[ e^{b \vert \X^1_t \vert^q}\right]
< +\infty \quad \textrm{and} \quad \sup_{N \geq 1} \sup_{0 \leq t
\leq T} \E\left[ e^{b \vert X^1_t \vert^q}\right] <+ \infty
$$
on the particle system or equivalent conditions on $p'$ for the
second estimate \eqref{eq:main2}. Observe again by weak
convergence in $N$ of the law of $(X^1_t, V^1_t)$ to the law $f_t$
of $(\X^1_t, \V^1_t)$ that the latter new moment control assumption is
stronger.
\end{itemize}

\subsubsection{Extensions to nonlinearly dependent diffusion coefficient}\label{subsec:vardiff}

Some researchers have recently argued that the diffusion
coefficient at a given point $x$ may depend on the neighbours of
the point to be considered \cite{DFT10,yates-etal}. More
precisely, they can depend on local in space averaged quantities
of the swarm, such as the averaged local density or velocity.
The averaged local density at $x$ in the particle system $(X^i_t,
V^i_t)$ for $1 \leq i \leq N$ is defined as
$$
\frac{1}{N} \sum_{j=1}^N \eta_{\varepsilon} (x-X^j_t);
$$
from which its corresponding continuous version is
$$
\int_{\rr^{2d}} \eta_{\varepsilon} (x-y) df_t(y,w).
$$
Here $\eta_{\varepsilon} (x) = \frac{1}{\varepsilon^d} \eta
\Big(\frac{x}{\varepsilon} \Big)$ where $\eta$ is a nonnegative
radial nonincreasing function with unit integral but non
necessarily compactly supported and $\varepsilon$ measures the
size of the interaction. The name of ``local average'' comes from
the smearing of choosing $\eta_{\varepsilon}$ instead of a Dirac
delta at $0$, which would be meaningless in the setting of a
particle system. Such a diffusion coefficient is considered
in~\cite{DFT10} with $\eta (x) = \frac{Z}{1+\vert x \vert^2}$ and
$\varepsilon = 1$, from the point of view of the long-time behaviour of solutions to the kinetic equation, not of the mean-field limit: there the particle system evolves according to
the diffusive Cucker-Smale model
$$
\begin{cases}
  d X_t^i = V_t^i dt\\
\displaystyle{dV_t^i = \sqrt{\frac{1}{N} \sum_{j=1}^N a( X^i_t - X^j_t )} \, dB_t^i -  \frac{1}{N} \sum_{j=1}^N a(X_t^i - X_t^j) (V_t^i-V_t^j) dt,} \qquad 1 \leq i \leq N\\
\end{cases}
$$
with $a(x) = \frac{Z}{1+\vert x \vert^2}$.

Other local quantities upon which the diffusion coefficient may
depend on is the averaged local velocity at $x$ defined as
$$
\bar u (x) :=\frac{1}{N} \sum_{j=1}^N V^j_t \, \eta_{\varepsilon}
(x-X^j_t)
$$
in the particle system, and
$$
\int_{\rr^{2d}} w \,  \eta_{\varepsilon} (x-y) df_t(y,w)
$$
in the continuous setting. More generically, we can consider
diffusion coefficients in the particle system such as
\begin{equation}\label{eq:localav}
g \left( \frac{1}{N} \sum_{j=1}^N h(V^j_t) \,  \eta_{\varepsilon}
(x-X^j_t) \right).
\end{equation}
Here, $\eta_{\varepsilon}$ controls which individuals we should
take into account in the average and with which strength; among
these $X^j$, how each velocity influences at $x$ is controlled by
$h$; finally, after averaging over those $j$, $g$ controls how we
should compute the diffusion coefficient.

For instance, given the diffusion coefficient
$g(\overline{u}(x))$, we could argue that $g$ should be large for
small $\overline{u}(x)$ (large noise for small velocity), and
conversely; there we see $g$ as an even function, nonincreasing on
$\rr^+$. Similar coefficients were used in \cite{yates-etal} of
the form
$$
g \left( \frac{1}{N(x)} \sum_{j=1}^{N(x)} V^j_t \,
\eta_{\varepsilon} (x-X^j_t) \right)
$$
with $\eta (x) = Z \1_{\vert x \vert \leq 1}$ and $N(x) =
\sharp\{j; \vert x - X^j_t \vert \leq \varepsilon\}$. However, we
cannot include this scaling in the mean-field setting. 

On the other hand, mean-field limits such as those in Theorem
\ref{thm:main} have been obtained in~\cite{Mel96,Szn91} with the
diffusion coefficient
$$
\frac{1}{N} \sum_{j=1}^N \sigma(x,v; X^j_t, V^j_t)
$$
where $\sigma$ is a $2d \times 2d$ matrix with globally Lipschitz coefficients.

\medskip

We include the two variants above by considering diffusion coefficients of the form
$$
\sigma[X^i_t, V^i_t; \hat{f}^N_t]
$$
where, for a probability measure $f$ on $\rr^{2d}$, $\sigma[z; f]$ is a $2d \times 2d$ matrix with coefficients
$$
\sigma_{kl} [z; f] = g \Big( \int_{\rr^{2d}} \sigma_{kl} (z, z') \, df(z') \Big)
$$
in the notation $z = (x,v), z' = (x',v') \in \rr^{2d}.$ We shall assume that $g$ is globally Lipschitz on $\rr$ and
\begin{multline*}
\vert \sigma_{kl} (z, z') - \sigma_{kl} (\bar{z}, \bar{z}') \vert
\\
\leq
C \big( \min \{\vert x - \bar{x} \vert + \vert x' - \bar{x}' \vert , 1 \} + \vert v - \bar{v} \vert + \vert v' - \bar{v}' \vert  \big) (1 + \vert v \vert^q + \vert v'  \vert^q +   \vert \bar{v} \vert^q + \vert \bar{v}'  \vert^q.
 \big)
 \end{multline*}

In this notation, \cite{Szn91} corresponds to $g(x) = x$ and $\sigma_{kl}$ bounded and Lipschitz, and~\eqref{eq:localav} to $\sigma_{kl} (z,z') = h(v') \, \eta_{\varepsilon} (x-x')$ where the kernel $\eta_{\varepsilon}$ is bounded and Lipschitz and
$$
\vert h(v) - h(v') \vert \leq C \vert v - v' \vert (1 + \vert v \vert^q + \vert v'  \vert^q).
$$
Observe that this framework does not include the model considered in~\cite{DFT10} for which the diffusion coefficient is given by a non locally Lipschitz $g$.

In this notation and assumption, if furthermore there exists $b>0$ such that
$$
\sup_{0 \leq t \leq T} \E\left[ e^{b \vert \X^1_t \vert^{2q}}\right]
< +\infty \quad \textrm{and} \quad \sup_{N \geq 1} \sup_{0 \leq t
\leq T} \E\left[ e^{b \vert X^1_t \vert^{2q}}\right] <+ \infty
$$
on the nonlinear process and particle system, then~\eqref{eq:main1} holds in Theorem~\ref{thm:main}, and correspondingly with $p'$  for the second estimate \eqref{eq:main2} (see Remark~\ref{rem:vardiff}).

\

\subsubsection{One-variable formulation}

We now give a formulation of the mean-field limit in one variable $z
\in \rr^D$, to be thought of as $z = (x,v) \in \rr^{2d}$ as in our
examples above or as $z = v \in \rr^d$ in a space-homogeneous
setting. We consider the particle system
$$
dZ^i_t = \sigma \, dB^i_t - F(Z^i_t) \, dt - \frac{1}{N}
\sum_{j=1}^N H(Z^i_t - Z^j _t) \, dt, \quad 1 \leq i \leq N
$$
where $\sigma$ is a (for instance) constant $D \times D$ matrix, the $(B^i_t)_{t
\geq 0}$ are $N$ independent standard Brownian motions on $\rr^D$
and the initial data $Z^i_0$ are independent and identically
distributed. We also consider the nonlinear processes $(\Z^i_t)_{t
\geq 0}$ defined by
$$
\begin{cases}
  d\Z_t^i= \sigma dB_t^i - F(\Z_t^i) dt -  H \ast f_t (\Z_t^i) dt,\\
  \Z_0^i = Z_0^i, \; f_t = \textrm{law}(Z_t^i).
\end{cases}
$$
Assume now that there exists $C>0$ such that
$$
(z-z') \cdot (F(z) - F(z')) \geq - C \vert z - z' \vert^2 (1 +
\vert z \vert^p + \vert z' \vert^p)
$$
for all $z, z' \in \rr^D$ with $p>0$.  Assume also global
existence and uniqueness of these processes, with
$$
\sup_{N \geq 1} \sup_{0 \leq t \leq T} \E\left[e^{b \vert Z^1_t
\vert^p}\right] <+ \infty
$$
and
$$
\sup_{0 \leq t \leq T} \left\{ \int_{\rr^D} e^{b \vert z \vert^p}
df_t(z) + \int_{\rr^{2D}} \vert H(z-z') \vert^2 df_t(z) \,
df_t(z')\right\} <+ \infty,
$$
or equivalent conditions on $p'$. Then \eqref{eq:main1} and
\eqref{eq:main2} in Theorem~\ref{thm:main} holds.


\section{Mean-field limit: proof}
\label{sec:limit}

This section is devoted to the proof of Theorem~\ref{thm:main}. We
follow the coupling method \cite{Szn91,Mel96,villani-champmoyen}. Given
$T>0$, we will use $C$ to denote diverse constants depending on $T$, the
functions $F$ and $H$, and moments of the solution $f_t$ on $[0,T]$,
but not on the number of particles $N$. 

\begin{proof}[Proof of Theorem \ref{thm:main}] Let us define the
  fluctuations as $x^i_t := X^i_t - \X^i_t$, $v^i_t := V^i_t -
  \V^i_t$, $i = 1,\dots,N$. For notational convenience, we shall drop
  the time dependence in the stochastic processes. As the Brownian
  motions $(B^i_t)_{t \geq 0}$ considered in \eqref{eq:sdesys} and
  \eqref{eq:nlSDE} are equal, for all $i=1, \dots, N$, we deduce
\begin{align}
  dx^i = \,&v^i \, dt \, ,\label{eq:CS-vidot0}\\
   dv^i = \, & - \big( F(X^i, V^i) - F(\X^i, \V^i) \big) dt \nonumber\\
   \, & -
\frac{1}{N}\sum_{j=1}^N \left( H(X^i \! - \! X^j, V^i-V^j)
    - (H*f_t)(\X^i,\V^i) \right) dt.   \label{eq:CS-vidot}
\end{align}
Let us consider the quantity $\alpha(t) = \E \left[ |x^i|^2
  +|v^i|^2\right]$ (independent of the label $i$ by symmetry), which
bounds the distance $W_2^2(f_t^{(1)},f_t)$ as remarked in
\eqref{prelim1}. Then, by using
\eqref{eq:CS-vidot0}-\eqref{eq:CS-vidot}, we readily get
\begin{equation}\label{eq:diffxi}
    \frac{1}{2} \frac{d}{dt} \E\left[|x^i|^2\right]
  =  \E\left[ x^i \cdot v^i\right] \leq   \frac{1}{2} \alpha(t)
\end{equation}
and
\begin{multline}
  \label{eq:diffvi}
  \frac{1}{2}\frac{d}{dt} \E\left[|v^i|^2\right]
  =
  - \E\left[ v^i \cdot \big(F(X^i, V^i) - F(\X^i,
    \V^i) \big) \right]
  \\
  \, -\frac{1}{N}
  \E\left[ \sum_{j=1}^N
    v^i \cdot
    \left( H(X^i \!  -  \! X^j, V^i \! - \! V^j)
      - H*f_t(\X^i, \V^i) \right)\right]
  =: I_1 + I_2.
\end{multline}
  
{\it Step 1.- Estimate $I_1$ by moment bounds:} We decompose $I_1$
in \eqref{eq:diffvi} as
$$
I_1=-\E\left[ v^i \cdot \big(F(X^i, V^i) - F(X^i, \V^i) \big)
\right] - \E\left[ v^i \cdot \big(F(X^i, \V^i) - F(\X^i, \V^i)
\big) \right].
$$
By assumption \eqref{hypaf1}-\eqref{hypaf2} on $F$, $I_1$ can
be controlled by
$$
I_1\leq A \, \E\left[ | v^i |^2 \right] + L \, \E \left[  | v^i |
\, \min\{| x^i |, 1\} \,  (1 + | \V^i |^p) \right]:=
I_{11}+L\, I_{12}\, .
$$
Given $R>0$, the second term $I_{12}$ is estimated according to
\begin{align*}
I_{12}\leq \, & \, \E [ | v^i | \, | x^i | ] +  \E \! \left[ \1_{| \V^i| \leq R} \, | v^i | \, \min\{|
x^i |, 1\} \, | \V^i |^p \right] \! + \E \! \left[ \1_{| \V^i| >
R} \,| v^i | \,  \min\{| x^i |, 1\}  \,   | \V^i |^p \right]
\\
\leq\,& (1+ R^p) \E\left[ | v^i | \, | x^i | \right]
 + \frac{1}{2} \E\left[ | v^i |^2 \right] + \frac{1}{2} \E \left[ \1_{| \V^i| > R} \,   | \V^i |^{2p} \right]
 \\
\leq\,& (1 + R^p) \alpha(t) + \frac{1}{2} \left( \E\left[  |
\V^i |^{4p} \right] \right)^{1/2} \left( \E \left[ \1_{| \V^i| > R}
\right] \right)^{1/2}
\end{align*}
by the Young and the Cauchy-Schwarz inequalities. Invoking the Markov
inequality, hypothesis \eqref{hyp:unifmoment} implies that there
exists $C>0$ such that
\begin{equation}
  \label{eq:markovexp}
  \E \left[ \1_{| \V^i_t| > R} \right] \leq e^{-a R^p} \,\E\left[
    e^{a | \V^i_t |^p}\right] \leq C \, e^{-a \, R^p}
\end{equation}
for all $i$ and $0\leq t \leq T$. By defining $r= a R^p /2$, we
conclude that given $T>0$, there exists $C>0$ such that
\begin{equation}\label{concl1}
I_1\leq C (1+r) \, \alpha(t) + C \, e^{- r}
\end{equation}
holds for all $r>0$ and all $0\leq t \leq T$.

\medskip

{\it Step 2.- Estimate $I_2$ by moment bounds:} We decompose the
second term in \eqref{eq:diffvi} as
\begin{equation}
  \label{eq:twotermsH}
  \begin{split}
    I_2 = \,&-\frac{1}{N} \E\left[ \sum_{j=1}^N v^i \cdot \left( H(X^i
        \! - \! X^j, V^i-V^j) - H(\X^i \!  - \!  \X^j,\V^i \! - \!
        \V^j) \right)\right]
    \\
    \,&- \frac{1}{N} \E\left[ v^i \cdot \left(H(0,0)-
        (H*f_t)(\X^i,\V^i) \right)\right]
    \\
    \,&- \frac{1}{N} \E\left[ \sum_{j\neq i}^N v^i
      \cdot \left(H(\X^i \! - \!  \X^j, \V^i \! - \! \V^j)
        - (H*f_t)(\X^i,\V^i) \right)\right]
    \\
    =: \,& I_{21}+I_{22}+I_{23}.
  \end{split}
\end{equation}
Since all particles are equally distributed and $H$ is antisymmetric,
we rewrite $I_{21}$ as
$$
I_{21} = - \frac{1}{2N^2}
   \sum_{i,j=1}^N
    \E\left[(v^i-v^j) \cdot \left( H(X^i  \! -  \! X^j, V^i  \! -  \! V^j)
    - H(\X^i  \! -  \! \X^j, \V^i  \! -  \! \V^j) \right) \right].
$$
Analogously to the argument used to bound $I_1$ in the first step, for
each $(i,j)$ we introduce the intermediate term $H(X^i-X^j,
\V^i-\V^j)$, split the expression in two terms, and estimate the
corresponding expectations using \eqref{hypaf1}-\eqref{hypaf2} on $H$
by
\begin{equation}\label{eq:Hij2terms}
I_{21}\leq A \, \E\left[ | v^i - v^j |^2 \right]+ L \, \E \left[ |
v^i - v^j | \, \min\{ | x^i - x^j |, 1 \} \, ( 1 + | \V^i - \V^j
|^p) \right]\, .
\end{equation}
For a given $R > 0$, and fixed $(i,j)$, consider the event $
\mathcal{R} := \{ |\V_i| \leq R, \ |\V_j| \leq R \}$ and the
random variable $\mathcal{Z}:=| v^i - v^j | \, \min\{ | x^i - x^j
|, 1 \} \, ( 1 + | \V^i - \V^j |^p)$. Then the last expectation
in~\eqref{eq:Hij2terms} can be estimated as follows, using again the
Young and Cauchy-Schwarz inequalities:
\begin{align}
  \E \left[ \mathcal{Z} \right]
  =\,& \E \left[ \1_{\mathcal{R}}
    \mathcal{Z} \right] + \E \left[ \1_{\mathcal{R}^C} \mathcal{Z}
  \right]
  \nonumber
  \\
  \leq \, & (1 + 2^p R^p) \, \E\left[ |v^i-v^j| \, | x^i - x^j
    |\right] + \frac{1}{2} \E\left[ | v^i - v^j |^2 \right]
  \nonumber
  \\
  \,&+ \E \left[ \1_{\mathcal{R}^C} \, ( 1 + | \V^i - \V^j |^p)^{2}
  \right]
  \nonumber
  \\
  \leq \,& 2 (1 + 2^p R^p) \, \alpha(t) + 2 \alpha(t) +  \left(\E
    \left[\1_{\mathcal{R}^C}\right] \right)^{1/2}
  \; \left(\E\left[
      (1 + |\V^i-\V^j|^p)^{4} \right] \right)^{1/2}
  \nonumber
  \\
  \leq \,& 2 (2 + 2^p R^p) \, \alpha(t)
  + C \left(\E \left[\1_{| \V^i
        | >R}\right]  + \E \left[\1_{| \V^j | >R}\right]
  \right)^{1/2}
  \;
  \left( 1 + \E\left[ | \V^i |^{4p} \right] \right)^{1/2}
  \nonumber
  \\
  \leq \,& C (1+R^p) \alpha(t) + C  \,  e^{-a \, R^p/2}\label{tech1}
\end{align}
by hypothesis~\eqref{hyp:unifmoment}. Inserting \eqref{tech1} into
\eqref{eq:Hij2terms} and defining $r= a R^p/2$, we conclude that given
$T>0$, there exists $C>0$ such that
\begin{equation}\label{concl21}
I_{21}\leq C (1+r) \, \alpha(t) + C \, e^{- r}
\end{equation}
holds for all $r>0$ and all $0\leq t \leq T$.

We now turn to estimate $I_{22}$, i.e., the second term
in~\eqref{eq:twotermsH}. Using that $H(0,0)=0$, we get
\begin{equation}\label{tech4}
  I_{22}
  \leq
  \frac{1}{N} \left(
    \E\left[|v^i|^2\right]
  \right)^{1/2}
  \left( \E\left[ \big| (H*f_t)(\X^i,\V^i) \big|^2\right]
  \right)^{1/2}
  \leq
  \frac{C}{N} \sqrt{\alpha(t)}\, .
\end{equation}
The latter inequality follows from
\begin{equation}\label{eq:estintH}
\E\left[ \big| (H*f_t)(\X^i,\V^i) \big|^2 \right] =
\int_{\rr^{4d}} | H(x-y, v-w) |^2 \, df_t(x,v) \, df_t(y,w),
\end{equation}
which is bounded on $[0,T]$ due to hypothesis
\eqref{hyp:unifmoment}.

The last term $I_{23}$ is treated as in the classical case in
\cite[Page 175]{Szn91} by a law of large numbers argument. We include
here some details for the sake of the reader. By symmetry we assume
that $i=1$. We start by applying the Cauchy-Schwarz inequality to
obtain
$$
 I_{23} \leq \frac{1}{N}
  \left( \E\left[|v^1|^2\right] \right)^{1/2}
  \Bigg(
    \E\left[ \Big|
    \sum_{j=2}^N
   Y^j \Big|^2\right]
  \Bigg)^{1/2}
$$
where $Y^j := H(\X^1 \! - \! \X^j, \V^1 \! - \! \V^j) -
(H*f_t)(\X^1,\V^1)$ for $j \geq 2$. Note that, for $j \neq k$,
$$
\E\left[ Y^j \cdot Y^k \right] = \E\left[ \E\left[  Y^j \cdot Y^k
| (\X^1, \V^1) \right] \right] = \E\left[ \E\big[ Y^j | (\X^1,
\V^1) \big]  \cdot  \E \big[ Y^k | (\X^1, \V^1) \big] \right]
$$
by independence of the $N$ processes $(\X^j_t, \V^j_t)_{t \geq
0}$, where
$$
\E\left[ Y^j | (\X^1, \V^1) \right]
  = \int_{\R^{2d}}
  [ H(\X^1_t -y,\V^1_t - w) - (H*f_t)(\X^1_t,\V^1_t) ] \, f_t(y,w) \,dy \,dw =0
$$
since $(\X^j_t, \V^j_t)$ has probability distribution $f_t$.
Hence,
\begin{align*}
  \E\left[ \Big|
  \sum_{\substack{j=2}}^N
  Y^j
  \Big|^2\right]
  &= (N-1) \E\left[ |Y^{2}|^2 \right] \\
  &\leq (N-1) \int_{\rr^{4d}} | H(x-y, v-w) |^2 \, df_t(x,v) \, df_t(y,w) \leq C \,
  (N-1)
\end{align*}
as in~\eqref{eq:estintH} due to hypothesis \eqref{hyp:unifmoment}.
Therefore, we get
\begin{equation}\label{tech2}
I_{23}\leq \frac{C}{\sqrt{N}} \, \sqrt{\alpha(t)} \,.
\end{equation}
Hence, combining the estimates \eqref{concl21}, \eqref{tech4} and
\eqref{tech2} to estimate $I_2$, we get that there exists $C>0$
such that
\begin{equation}\label{concl2}
I_{2}\leq C (1+r) \, \alpha(t) + C \, e^{- r} + \frac{C}{\sqrt{N}}
\, \sqrt{\alpha(t)}
\end{equation}
holds for all $r>0$ and all $0\leq t \leq T$.

\medskip

{\it Step 3.- Proof of \eqref{eq:main1}:} It follows
from~\eqref{eq:diffxi},~\eqref{eq:diffvi},~\eqref{concl1},~\eqref{concl2},
the above estimates and the Young inequality that
$$
  \alpha'(t)
  \leq
  C \left(1 + r \right)
  \alpha(t)  + C e^{-r}
  +
  \frac{C}{\sqrt{N}} \sqrt{ \alpha(t) }
    \leq
  C \left(1 + r \right)
  \alpha(t)  + C e^{-r}
  +
  \frac{C}{N}
$$
for all $t \in [0,T]$, all $N \geq 1$ and all $r>0$. From this
differential inequality and Gronwall's lemma, we can first deduce that
the quantity $\alpha(t)$ is bounded on $[0,T]$, uniformly in $N$, by a
constant $D > 0$. Hence, the function $\beta(t) := \alpha(t) /(eD)$ is
bounded by $1/e$, so that $1 - \ln \beta \leq - 2 \ln \beta$. Now,
whenever $\beta(t) > 0$, take $r := - \ln \beta(t) > 0$. This choice
proves that, for any $t$ such that $\beta(t) > 0$,
\begin{equation}
  \beta'(t) \leq C (1 - \ln \beta(t) ) \, \beta(t) + \frac{C}{N}
  \leq -C \, \beta (t) \, \ln \beta(t) + \frac{C}{N}.\label{eq:diff-ineq-beta}
\end{equation}
Actually, the above inequality is also true whenever $\beta(t) = 0$
(with the convention that $z \log z = 0$ for $z = 0$), as can be seen
by choosing $r := \log N$ in that case. Hence,
\eqref{eq:diff-ineq-beta} holds for all $t \in [0,T]$. Now, the function
$u(t) := \beta(Ct)$ satisfies $u(0) = 0$ and
$$
u' \leq - u \, \ln u + \frac{1}{N}
$$
on $[0,T/C]$. Let finally $a(t)$ be a function on $[0,T/C]$, to be
chosen later on. Then the map $v(t) = u(t) \, N^{a(t)}$ satisfies
$v(0) = 0$ and
$$
v' \leq - v \ln v + N^{a-1} + v \, \ln N (a + a') \leq - v \ln v +
1 \leq \frac{1}{e} + 1
$$
on $[0,T/C]$ provided we choose $a(t) = e^{-t} \leq 1$. Hence,
this choice of $a(t)$ implies the bound
$$
v(t) \leq \Big(\frac{1}{e} +1 \Big) \frac{T}{C}
$$
for $0 \leq t \leq T/C$, that is,
$$
\ee \big[ | X^i_t - V^i_t |^2 + | \X^i_t - \V^i_t |^2 \big] =
\alpha(t) \leq C N^{-e^{-Ct}}
$$
for $0 \leq t \leq T$, and thus, \eqref{eq:main1} is proven.

\medskip

{\it Step 4.- Proof of \eqref{eq:main2}:} If additionally there
exists $p'>p$ such that hypothesis \eqref{hyp:unifmoment2} holds,
then by the Markov inequality, estimate~\eqref{eq:markovexp} turns
into
$$
\E\left[ \1_{| \V^i_t| > R} \right] \leq e^{-a R^{p'}} \E\left[
e^{a | \V^i_t |^{p'}}\right] \leq C \, e^{-a \, R^{p'}}.
$$
Hence, following the same proof, the quantity $\alpha(t)$ finally
satisfies the differential inequality
$$
\alpha'(t) \leq C (1+r) \, \alpha(t) + C \, e^{- r^{p'/p}} +
\frac{C}{N}
$$
for all $N \geq 1$ and all $r >0$. If we choose $r = (\ln N)^{p/p'}$,
and since $\alpha (0) = 0$, this integrates to
$$
\alpha(t)
\leq
\frac{2}{N(1+r)} \Big( e^{C(1+r)t} - 1 \Big)
\leq
\frac{2}{N} \, e^{C(1+r)T}
=
2 \, e^{CT} \, e^{C (\ln N)^{p/p'} T - \ln N}.
$$
Given $\epsilon >0$, there exists a constant $D$ such that
$$
C (\ln N)^{p/p'} T - \ln N \leq D - (1 - \epsilon) \ln N
$$
for all $N \geq 1$, so that
$$
\alpha (t) \leq 2 \, e^{CT+D} \, N^{-(1-\epsilon)}
$$
for all $N \geq 1$. This concludes the proof of
Theorem~\ref{thm:main}.
\end{proof}

\begin{rem}\label{rem:vardiff}
In the  setting of subsection~\ref{subsec:vardiff} when the constant diffusion coefficient $\sqrt{2}$ governing the evolution of the particle $(X^i_t, V^i_t)$ is replaced by a more general 
$\sigma[X^i_t, V^i_t; \hat{f}_t],$
then, by the It\^o formula, we have to control the extra term
$$
\sum_{k,l} \ee \big[ \big\vert \sigma_{kl}[X^i_ t, V^i_t; \hat{f}^N_t] - \sigma_{kl}[\X^i_t, \V^i_t; f_t] \big\vert^2
\big].
$$
For that purpose we use the Lipschitz property of $g$, introduce the intermediate term $\displaystyle \frac{1}{N} \sum_{j=1}^N \sigma_{kl} \big((\X^i_t, \V^i_t),(\X^j_t, \V^j_t) \big)$ and adapt the argument used above to bound the term $I_2$.
\end{rem}


\section{Existence and uniqueness}
\label{sec:existence}

This section is devoted to the proof of Theorem \ref{thm:mainexist} on
existence, uniqueness and propagation of moments for solutions of the
particle system~\eqref{eq:sdesys}, the nonlinear
process~\eqref{eq:nlSDE} and the associated PDE~\eqref{eq:pde}. This
provides a setting under which Theorem \ref{thm:main} holds, showing
that the existence and moment bound hypotheses are satisfied under
reasonable conditions on the coefficients of the equations and the initial data alone.

Given $T>0$, we will denote by $b$ and $C$ constants, that may change
from line to line, depending on $T$, the functions $F$ and $H$, and
moments of the initial datum $f_0$.

\subsection{Existence and uniqueness of the particle system}

Let us start by proving point \emph{i)} of Theorem
\ref{thm:mainexist}. In this section we let $f_0 \in \mathcal
P_2(\rr^{2d})$ and consider the particle system for $1 \leq i \leq N$:
\begin{equation}\label{eq:sdesyst}
\begin{cases}
  d X_t^i = V_t^i dt\\
\displaystyle{dV_t^i = \sqrt{2} dB_t^i - F(X_t^i, V_t^i) dt
- \frac{1}{N} \sum_{j=1}^N H(X_t^i - X_t^j, V_t^i-V_t^j) dt,} \\
\end{cases}
\end{equation}
with initial data $(X^i_0, V^i_0)$ for $1 \leq i \leq N$
distributed according to $f_0$. Here the $(B^i_t)_{t \geq 0}$, for
$i=1,\dots,N$, are $N$ independent standard Brownian motions on
$\rr^d$.

\begin{lem}\label{lem:existsde}
  Let $f_0 \in \mathcal P_2(\rr^{2d})$, and assume that $F$, $H$ are
  locally Lipschitz and satisfy \eqref{eq:F-one-sided-growth} and
  \eqref{eq:H-two-sided-sys}. For $1 \leq i \leq N$, take random
  variables $(X^i_0, V^i_0)$ with law $f_0$. Then~\eqref{eq:sdesyst}
  admits a pathwise unique global solution with initial datum $(X^i_0,
  V^i_0)$ for $1 \leq i \leq N$.
\end{lem}

\begin{proof}
The system~\eqref{eq:sdesyst} can be written as the SDE
$$
d{\bZ^N_t} = \sigma^N
\, d{\bB^N_t} + {\bb} ({\bZ^N_t})
\, dt
$$
in $\rr^{2dN}$, where ${\bZ^N_t} = (X^1_t, V^1_t, \dots,
X^N_t,V^N_t)$. Here $\sigma^N$ is a constant $2dN \times 2dN$ matrix,
$({\bB^N_t})_{t \geq 0}$ is a standard Brownian motion on $\rr^{2dN}$,
and $\bb:\R^{2dN} \to \R^{2dN}$ is a locally Lipschitz function
defined in the obvious way. Moreover, letting $\ap{ \cdot, \cdot}$ be
the scalar product and $\Vert \cdot \Vert$ the Euclidean norm on
$\rr^{2dN}$, then for all ${\bZ^N} = (X^1, V^1, \dots, X^N, V^N)$,
\begin{align*}
    \ap{ {\bZ^N}, {\bb} ({\bZ^N}) }
    =&\,
    \sum_{i=1}^N X^i \cdot V^i
    - \sum_{i=1}^N V^i \cdot F(X^i, V^i)
    - \frac{1}{N} \sum_{i,j=1}^N V^i
    \cdot H(X^i \! - \! X^j, V^i \!- \! V^j)
    \\
    \leq&\, (C+\frac{1}{2}) ( N + \|{\bZ^N}\|^2)
  +  C  \frac{1}{N} \sum_{i,j=1}^N \vert V^i \vert (1 + \vert V^i - V^j \vert)
    \\
    \leq&\, C ( N + \|{\bZ^N}\|^2)
    +
    \frac{C}{2N} \sum_{i,j=1}^N (1 + |V^i|^2 + |V^j|^2)
    \leq
    C(N + \| {\bZ^N} \|^2).
\end{align*}
Here we have used the elementary inequality $2ab \leq a^2 + b^2$,
 and the bounds
\eqref{eq:F-one-sided-growth} and \eqref{eq:H-two-sided-sys}. This is
a sufficient condition for global existence and pathwise
uniqueness, see~\cite[Chapter 5, Theorems 3.7 and
3.11]{ethier-kurtz86} for instance.
\end{proof}

\begin{rem}
  For the existence we do not use any properties of symmetry of the
  system, and in particular we do not need the initial data to be
  independent.
  On the other hand, the condition \eqref{eq:H-two-sided-sys} in Lemma
  {\rm\ref{lem:existsde}} can be relaxed to
\begin{equation*}
    -v \cdot H(x,v) \leq C (1 + |v|^2),
\end{equation*}
if we impose that $H$ is antisymmetric, i.e., $H(-x,-v) = -H(x,v)$ for
all $x,v\in \R^d$. Actually, in this case we can perform a
symmetrization in $(i,j)$ to estimate the term involving $H$
by
\begin{align*}
  - \frac{1}{N} \sum_{i,j=1}^N V^i
  \cdot H(X^i \! - \! X^j, V^i \!- \! V^j)
  &=
  -
  \frac{1}{2N} \sum_{i,j=1}^N (V^i \! - \! V^j)
  \cdot H(X^i \! - \! X^j, V^i \!- \! V^j)
  \\
  &\leq
  \frac{C}{2N} \sum_{i,j=1}^N (1 + |V^i \!- \! V^j|^2)
  \leq
  C (N + \| {\bf Z^N} \|^2).
\end{align*}
\end{rem}

\subsection{Existence and uniqueness for the nonlinear process and PDE}
\label{sec:SDE-nonlinear}

In this section we prove points \emph{ii)} and \emph{iii)} in Theorem
\ref{thm:mainexist}, namely, the existence and uniqueness of solutions
to the nonlinear SDE \eqref{eq:nlSDE}:
\begin{equation*}
  \left\{
    \begin{split}
      &d\X_t=\V_t\,dt
      \\
      &d\V_t= \sqrt{2}\, dB_t- F(\X_t, \V_t)dt - H \ast f_t
      (\X_t,\V_t)dt,
      \\
      &f_t = \law(\X_t, \V_t), \quad \law(\X_0, \V_0) = f_0.
    \end{split}
  \right.
\end{equation*}
and to the associated nonlinear PDE \eqref{eq:pde}:
\begin{equation*}
  \partial_t f_t + v  \cdot \nabla_x f_t
  = \Delta_v f_t + \nabla_v \cdot ((F + H*f_t) f_t)\,,
  \quad t>0, \ x,v \in \rr^d\,,
\end{equation*}
under the hypotheses of Theorem \ref{thm:mainexist}. Notice that we
drop the superscript $i$ for the SDE \eqref{eq:nlSDE}, as the problem
is solved independently for each $i$. For the PDE \eqref{eq:pde},
we always consider solutions in the sense of distributions:
\begin{dfn}
  \label{dfn:solutionPDE}
  Assume that $F,H: \R^{2d} \to \R^{2d}$ are continuous, and that
  \eqref{eq:H-two-sided-sys} holds. Given $T > 0$, a function $f:[0,T]
  \to \P_2(\R^{2d})$, continuous in the $W_2$ topology, is a solution
  of equation \eqref{eq:pde} with initial data $f_0 \in \P_2(\R^{2d})$
  if for all $\varphi \in \mathcal{C}^\infty_0([0,T) \times \R^{2d})$
  it holds that
  \begin{equation}
    \label{eq:PDE-weak}
    \int_{\rr^{2d}} \varphi_0 \,df_0
    =
    -\int_0^T \!\!\! \int_{\rr^{2d}} \big( \partial_s \varphi_s + \Delta_v \varphi_s - \nabla_v \varphi_s \cdot
    (F + H * f_s) + \nabla_x \varphi_s \cdot v \big) \, df_s \, ds.
  \end{equation}
\end{dfn}

\

\noindent Notice that all terms above make sense due to the continuity
of $F,H$, equation \eqref{eq:H-two-sided-sys}, and the bound on the
second moment of $f_t$ on bounded time intervals (needed for $H*f_t$
to make sense). For the purpose of this definition, condition
\eqref{eq:H-two-sided-sys} can actually be relaxed to $|H(x,v)| \leq
C(1 + |x|^2 + |v|^2)$, but we will always work under the stronger
hypothesis below.

The proof spans several steps, that we split in several
subsections. Since some parts of the proof hold under weaker
conditions on $F$ and $H$, we will specify the hypotheses needed in
each part.

\subsubsection{Existence and uniqueness of an associated linear SDE}
\label{sec:linSDE}

Let us first consider a related linear problem: we want to solve the
SDE
\begin{equation}
  \label{eq:edslin}
  \begin{cases}
    d X_t = V_t dt
    \\
    dV_t = \sqrt{2} dB_t - F(X_t, V_t) dt -  (H*g_t) (X_t, V_t) dt
  \end{cases}
\end{equation}
for given $F, H:\R^{2d} \to \R^{d}$ and $g : [0,T] \to \P_2(\R^{2d})$.

\begin{lem}
  \label{lem:edslin}
  Assume that $F$ and $H$ are locally Lipschitz functions satisfying
  \eqref{eq:F-one-sided-growth}, \eqref{eq:H-two-sided-sys}, and
  \eqref{eq:H-loc-Lipschitz}. Let $f_0 \in \mathcal P_2(\rr^{2d})$ and
  $g:[0,T] \to \P_2(\R^{2d})$ be a continuous curve in the $W_2$
  topology. Then the equation \eqref{eq:edslin} with inital datum
  $(X_0, V_0)$ distributed according to $f_0$ has a global pathwise
  unique solution. Moreover this solution has bounded second moment on
  $[0,T]$.
\end{lem}

\begin{proof}
We rewrite~\eqref{eq:edslin} as
$$
  dZ_t = \sigma \, dB_t + b(t, Z_t) \, dt
$$
on $\rr^{2d}$, where $Z_t = (X_t, V_t)$. Here $\sigma$ is a $2d \times
2d$ matrix, $(B_t)_{t \geq 0}$ is a standard Brownian motion on
$\rr^{2d}$ and
$$
  b(t,x,v) = (v, - F(x,v) - (H*g_t)(x,v)).
$$
Let us first observe that the growth condition
\eqref{eq:H-two-sided-sys} on $H$ implies
\begin{equation}
    \label{eq:estH0}
    | (H*g_t) (x,v) | \leq
    C \!\int_{\rr^{2d}}\!\!\! (1 + |v-w|) \, dg_t(y,w)
\leq
    C \left(
      1 + |v| + \int_{\rr^{2d}} |w| \, dg_t(y,w)
    \right).
\end{equation}
This together with the Cauchy-Schwarz inequality results in
\begin{equation}
    \label{eq:estH}
    | v \cdot (H * g_t) (x, v) |
    \leq
    C \left(
      1 + \vert v \vert^2 + \int_{\rr^{2d}} \vert w \vert^2
      \, dg_t(y,w)
    \right)\,.
\end{equation}
Estimate \eqref{eq:estH0} ensures that, for fixed $x,v \in \R^d$, the
map $t \mapsto (H*g_t) (x,v)$ is bounded on $[0,T]$, using the fact
that the second moment of $g_t$ (and hence its first moment) is
uniformly bounded on $[0,T]$. Then, the same applies to $t \mapsto
b(t,x,v)$.

Also, the map $(x,v) \mapsto (H*g_t) (x,v)$ is locally Lipschitz,
uniformly on $t \in [0,T]$; indeed, for $x,v,y,w \in \R^d$, using
\eqref{eq:H-loc-Lipschitz} we get
\begin{align*}
  \big| (H*g_t)(x,v) &- (H*g_t) (y,w) \big|
  \\ \leq
     &\,
     \int_{\rr^{2d}} \Big\vert  H(x-X, v-V) - H(y-X, w-V) \Big\vert
     \, dg_t(X,V)
     \\
     \leq &\,
     L \, ( \big\vert x-y \big\vert + \big\vert v-w \big\vert )
     \int_{\rr^{2d}} (1+ \vert v-V \vert^p + \vert w-V \vert^p) dg_t(X,V)
     \\
    \leq &\,
    C ( \big\vert x-y \big\vert + \big\vert v-w \big\vert )
    \Big[ 1 + \vert v \vert^p + \vert w \vert^p
    + \int_{\rr^{2d}} \vert V \vert^p dg_t(X,V) \Big].
\end{align*}
The moment of $g_t$ above is bounded on $[0,T]$ since $p\leq 2$ and
the curve is continuous in the $W_2$-metric. As $F$ is also locally
Lipschitz, we conclude that $t \mapsto b(t,x,v)$ is locally Lipschitz.

Finally, for the scalar product $\ap{\cdot,\cdot}$ on $\rr^{2d}$,
we deduce
\begin{equation}\label{eq: driftxv}
  \ap{ (x,v) , b(t,x,v) }
  =
  x \cdot v - v \cdot F(x,v) - v \cdot (H*g_t) (x,v)
  \leq
  C \Big(1 + \vert x \vert^2+ \vert v \vert^2 \Big),
\end{equation}
by \eqref{eq:F-one-sided-growth}, \eqref{eq:estH}, and again the fact
that the second moment of $g_t$ is uniformly bounded on $[0,T]$. As in
the proof of Lemma~\ref{lem:existsde}, these are sufficient conditions
for global existence and pathwise uniqueness for solutions
to~\eqref{lem:edslin} with square-integrable initial data.

Moreover, by~\eqref{eq: driftxv},
$$
\frac{d}{dt} \ee [ \vert X_t \vert^2 + \vert V_t \vert^2 ]
=
2 d + 2 \, \ee \ap{ (X_t, V_t) , b(t, X_t , V_t) }
\leq
2d + C \ee \big[ 1 +  \vert X_t \vert^2+ \vert V_t \vert^2 \big],
$$
so that by integration the second moment $ \ee [ \vert X_t \vert^2 + \vert V_t \vert^2 ]$ is bounded on $[0,T]$.
\end{proof}

\subsubsection{Existence and uniqueness of an associated linear PDE}
\label{sec:linearPDE}

By It\^o's formula, the law $f_t$ of the solution of
\eqref{eq:edslin} at time $t$ is a solution of the following
linear PDE:
\begin{equation}
  \label{eq:linearPDE}
  \partial_t f_t + v \cdot \nabla_x f_t
  =
  \Delta_v f_t + \nabla_v \cdot (f_t \, (F + H * g_t))\,,
\end{equation}
in the distributional sense as in \eqref{eq:PDE-weak} of
Definition~\ref{dfn:solutionPDE}. Moreover, the curve $t \mapsto f_t$
is continuous for the $W_2$ topology.  Indeed, on the one hand
$$
W_2^2(f_t, f_s)
\leq
\ee \left[\vert X_t - X_s \vert^2 + \vert V_t
- V_s \vert^2\right].
$$
On the other hand, the paths $t \mapsto X_t (\omega)$ are continuous
in time for {\it a.e.} $\omega$, and $(X_t, V_t)$ has bounded second
moment on $[0,T]$; hence, by the Lebesgue continuity theorem, for
fixed $s$ the map $t \mapsto \ee \left[\vert X_t - X_s \vert^2 + \vert
  V_t - V_s \vert^2\right]$ is continuous, and hence converges to $0$
as $t$ tends to $s$. (Alternatively, one can obtain quantitative
bounds on the time continuity by estimating $\ee \left[\vert X_t - X_s
  \vert^2 + \vert V_t - V_s \vert^2\right]$ in the spirit of the last
equation in the proof of Lemma \ref{lem:edslin}; we do not follow this
approach here).

Then one can follow a duality argument in order to show that solutions
to \eqref{eq:linearPDE} are unique, which we sketch now. Take a
solution $f_t$ of \eqref{eq:linearPDE} with $f_0 = 0$; we wish to show
that $f_t = 0$ for any $t > 0$. For fixed $t_0 > 0$ and $\varphi$
smooth with compact support in $\rr^{2d}$, consider the solution
$h_t$ defined for $t \in [0,t_0]$ of the dual problem
\begin{gather*}
  \partial_t h_t + v \cdot \nabla_x h_t
  =
  - \Delta_v h_t + (F + H * g_t) \cdot \nabla_v h_t,
  \\
  h_{t_0} = \varphi.
\end{gather*}
This is a linear final value problem, and by considering $h_{t_0-t}$
one can show that it has a solution by classical arguments. In
addition, for each $t$, this solution $h_t$ is a continuous function,
as can be seen through classical results on propagation of
regularity. Then, as $h_t$ solves the dual equation of
\eqref{eq:linearPDE}, it holds that
\begin{equation*}
  \frac{d}{dt} \int f_t\, h_t = 0 \qquad (t \in (0,t_0)),
\end{equation*}
from which $\int f_{t_0} \varphi = \int f_{0} h_0 = 0$.  Since
$\varphi$ is arbitrary, this shows that $f_{t_0} = 0$ and proves the
uniqueness.

\subsubsection{Existence and uniqueness for the nonlinear PDE and SDE}

We are now ready to finish proving points \emph{ii)} and \emph{iii)}
in Theorem \ref{thm:mainexist}.

\paragraph{Step 1.- Iterative scheme:} Take $f^0 \in \mathcal P_2 (\R^{2d})$ and
random variables $(X^0,V^0)$ with law $f_0$, and let $(B_t)_{t
\geq
  0}$ be a given standard Brownian motion on $\rr^d$. We define the stochastic
processes $(X^n_t,V^n_t)_{t \geq 0}$ recursively by
$$
\begin{cases}
  dX^n_t=V^n_t\,dt \\
  dV^n_t =  \sqrt{2}\, dB_t - F(X^{n}_t, V^{n}_t)\,dt  -  (H * f^{n-1}_t) (X^{n}_t,V^{n}_t)\,dt,\\
  (X_0^{n},V_0^{n}) = (X^0, V^0)
\end{cases}
$$
for $n \geq 1$, where $f^n_t := \law(X_t^n,V_t^n)$ and it is
understood that $f^0_t := f^0$ for all $t \geq 0$. Observe that
these are linear SDEs for which existence and pathwise uniqueness
are given by Lemma~\ref{lem:edslin} since all $t \mapsto
f^{n-1}_t$ is continuous for the $W_2$ topology. We also know from
section \ref{sec:linearPDE} that the $f_t^n$  are weak solutions
to the PDE
\begin{equation*}
  \partial_t f^n_t + v \cdot \nabla_x f^n_t = \Delta_v f^n_t + \nabla_v \cdot ( f^n_t (F + H*f_t^{n-1})).
\end{equation*}
with initial condition $f_0$. More precisely, the following holds
for $n \geq 1$ and all $\varphi \in \mathcal{C}_0^{\infty} ([0,T)
\times \R^{2d})$:
\begin{equation}
  \label{eq:fn-PDE-weak}
  \int_{\rr^{2d}} \varphi_0 \, df^0 =
  - \int_0^T \int_{\rr^{2d}} ( \partial_s \varphi_s + \Delta_v \varphi_s - \nabla_v \varphi_s \cdot
  (F + H * f^{n-1}_s) + \nabla_x \varphi_s \cdot v ) \,df^n_s \, ds.
\end{equation}

\paragraph{Step 2.- Uniform estimates on moments of $f^n_t$:}

We shall prove the following lemma:
\begin{lem}\label{lem:momscheme}
Assume Hypothesis
\eqref{eq:F-one-sided-growth}--\eqref{eq:H-loc-Lipschitz} on $F$
and $H$. Let $f_0$ be a probability measure on $\rr^2$ such that
$$
  \int_{\rr^{2d}} \big( \vert x \vert^2 + e^{a |v|^p}\big)  \, df_0(x,v)
  < + \infty
$$
for a positive $a$. Then for all $T$ there exists a positive
constant $b$ such that, for the laws $f^n_t$ of the processes
$(X^n_t, V^n_t)$,
$$
  \sup_{n \geq 1} \sup_{0 \leq t \leq T}
 \int_{\rr^{2d}}   \big( \vert x \vert^2 + e^{b |v|^p} \big)  \,df^n_t(x,v)
  < + \infty.
$$
\end{lem}

\begin{proof}
We prove this lemma in two steps.

{\it Step 1.- Bound for moments of order 2:} Let
$$
  e_n(t) = \int_{\rr^{2d}} |v|^2 \,df^n_t(x,v)
$$
for $n \geq 1$ and $t \geq 0$. Using \eqref{eq:F-one-sided-growth}
and \eqref{eq:estH} applied to the measure $f^{n-1}_t$, we get
\begin{align*}
    e'_n(t) = &\, \frac{d}{dt} \int_{\R^{2d}} |v|^2 \,df^n_t
    =
    2d  - 2 \int_{\R^{2d}} v \cdot ((F + H*f^{n-1}_t) \,df^n_t(x,v)\\
    \leq&\,
    2d + 2 C \int_{\R^{2d}}
    \left(1+|v|^2
    + \int_{\R^{2d}} \vert w \vert^2 df^{n-1}_t(y,w)
    \right)
    \,df^n_t(x,v)
    \\
    \leq&\,
    C  \big(1 + e_n(t) + e_{n-1}(t) \big)\, ,
\end{align*}
for diverse constants $C$ depending on $F$ and $H$ but not on $t$
or $n$. Since moreover
$$
  e_n(0) = e_0 (t) = e_0(0) = \int_{\rr^{2d}} |v|^2 \,df_0(x,v),
$$
for all $t$ and $n$, then one can prove by induction that
$$
  \sup_{n \geq 0} \int_{\R^{2d}} |v|^2 df^n_t
  \leq \Big( D + \frac{1}{2} \Big) e^{2 Dt} - \frac{1}{2}
$$
where $D = \max \{C, e_0(0)\}$. Moreover the bound
$$
\frac{d}{dt} \int_{\rr^{2d}} |x|^2 \,df^n_t(x,v) = 2 \int_{\rr^{2d}} x \cdot v  \,df^n_t(x,v)
\leq \int_{\rr^{2d}} |x| ^2 \,df^n_t(x,v) + e_n(t)
$$
ensures that also $\displaystyle \int_{\rr^{2d}} |x|^2 \,df^n_t(x,v) $ is bounded on $[0,T]$, uniformly in $n$.
  \medskip

\noindent {\it Step 2.- Bound for exponential moments:} Let
$\alpha= \alpha(t)$ be a smooth positive function to be chosen
later on and $\ap{v} = (1 + \vert v \vert^2)^{1/2}$. Then we have
the following \emph{a priori} estimate:
\begin{align*}
    \frac{d}{dt} \int_{\R^{2d}} e^{\alpha(t) \ap{v}^p} df^n_t(x,v)
    =
    \int_{\R^{2d}} \big[ d \alpha p \ap{v}^{p-2} + \alpha p (p-2) \vert v
    \vert^2 \ap{v}^{p-4} + \alpha^2 p^2 \vert v \vert^2 \ap{v}^{2p-4}
    \\
    + \alpha' \ap{v}^p - \alpha p \ap{v}^{p-2} v \cdot  (F +
    H*f^{n-1}_t) \big]  e^{\alpha \ap{v}^p} df^n_t(x,v).
\end{align*}
But, by \eqref{eq:F-one-sided-growth}, \eqref{eq:estH} applied to
$f^{n-1}_t$, and the bound on the moment of order $2$ in Step $1$,
$$
  -  v \cdot  (F +   H*f^{n-1}_t)
  \leq
  C \left(1+|v|^2
  + \int_{\R^{2d}} \vert w \vert^2 df^{n-1}_t(y,w) \right)
  \leq C \ap{v}^2
$$
uniformly on $n \geq 1$ and $t \in [0,T]$, so
\begin{align*}
  \frac{d}{dt} \int_{\R^{2d}} &\,e^{\alpha(t) \ap{v}^p} df^{n}_t(x,v)
  \\
  &\leq
  \int_{\R^{2d}} \big[ C \alpha  \ap{v}^{p-2}  + C  \alpha^2   \ap{v}^{2p-2} +  \alpha '(t) \ap{v}^p + C \alpha \ap{v}^{p}  \big]  e^{\alpha \ap{v}^p}
  df^n_t(x,v)\\
  &\leq \int_{\R^{2d}} \big[ C \alpha + C \alpha^2 + \alpha' \big] \ap{v}^p e^{\alpha \ap{v}^p} df^n_t(x,v).
\end{align*}
since $p \leq 2$. Choosing $\alpha$ such that $C \alpha + C \alpha^2 +
\alpha' \leq 0$ (for instance, $\alpha(t) = M e^{-2 C t}$, with $0 < M
\leq 1$) we conclude that
$$
\frac{d}{dt} \int_{\R^{2d}} e^{\alpha(t) \ap{v}^p} df^n_t \leq 0 \,
$$
Hence, we obtain
$$
  \int_{\R^{2d}} e^{\alpha(t) \ap{v}^p} df^n_t(x,v)
  \leq
  \int_{\R^{2d}} e^{\alpha(0) \ap{v}^p} df^n_0(x,v) = \int_{\R^{2d}} e^{\alpha(0) \ap{v}^p}
  df_0(x,v),
$$
which is finite provided $\alpha(0) \leq a$, which can be satisfied by
taking $M = \min\{a,1\}$ above. Then, taking $b = \alpha(T)$ we
conclude that
\begin{equation}
  \label{uniforminnbounds}
  \sup_{n\geq 0} \sup_{0 \leq t \leq T}   \int_{\R^{2d}} e^{b \vert v \vert^p}
  df^n_t <\infty.
\end{equation}
\end{proof}

We notice for later use that as a direct consequence of
\eqref{uniforminnbounds} using the Markov inequality, there exists $C
\geq 0$ such that
\begin{equation}\label{uniftailinnbounds}
 \sup_{0 \leq t \leq T} \sup_{n\geq 0} \E \left[ \1_{| V^n_t| > R} \right]
  \leq e^{-b R^p} \sup_{0 \leq t \leq T} \sup_{n\geq 0} \E \left[ e^{b | V^n_t |^p}\right] \leq C \, e^{-b \,
  R^p}\, .
\end{equation}

\paragraph{Step 3.- Existence for the nonlinear PDE.} We intend to carry out
an argument analogous to the one for the existence and uniqueness
of solutions for the 2D Euler equation in fluid mechanics, found
for example in \cite{Marchioro-Pulvirenti}. We will prove that the
$f_t^n$ converge to a limit, and that this limit is a solution to
the nonlinear PDE. To simplify notation we drop time subscripts
and use the following shortcuts:
$$
    v^{n} := V^{n+1} - V^n,
    \quad
    x^{n} := X^{n+1} - X^n,
    \quad
    Z^n := (X^n, V^n),
    \quad
    z := (x,v).
$$
Also, we write
$$
    \gamma^n(t) := \E\left[|x^n|^2\right] + \E\left[|v^n|^2\right]
    = \E \left[|Z^{n+1} - Z^n|^2\right].
$$
We compute, by It\^o's formula, and for any $n \geq 1$,
\begin{equation}
    \label{eq:ex0.5}
    \frac{d}{dt} \E\left[|x^n|^2\right]
    = 2\, \E \left[ x^n \cdot v^n \right]
    \leq
    \E \left[|x^n|^2\right] + \E\left[|v^n|^2\right] = \gamma^n(t)\,,
\end{equation}
with
\begin{align}
    \frac{1}{2} \frac{d}{dt} \E\left[|v^n|^2\right]
    =&\,
    - \E \left[ v^n \cdot \big(F(Z^{n+1}) - F(Z^n)\big)  \right]
    \nonumber\\
    &-
    \E \left[ v^n \cdot \Big((H*f^n_t)(Z^{n+1}) - (H*f^{n-1}_t)(Z^{n})\Big) \right]
    =: T_1 + T_2.    \label{eq:ex1}
\end{align}

{\it Estimate for $T_1$.} We decompose the term $T_1$ as
\begin{align*}
T_1 =&\,  -\E \left[(V^{n+1}  \!  -  \! V^n) \cdot \big(F(X^{n+1},
V^{n+1}) - F(X^{n+1}, V^n) \big)  \right]\\
  &\,-
  \E \left[ (V^{n+1}  \!  -  \! V^n) \cdot \big(F(X^{n+1}, V^n) - F(X^n, V^n) \big) \right],
\end{align*}
which by \eqref{eq:F-loc-Lipschitz-v}--\eqref{eq:F-loc-Lipschitz-x} is
bounded above by
\begin{align*}
    T_1&\leq L \, \E\left[ | v^n |^2(1 + |V^n|^p+ |V^{n+1}|^p)\right]
    + L \, \E  \left[| v^n | \, | x^n | \,  (1 + | V^n
    |^p)\right]\nonumber \\
    &=: T_{11} + T_{12}.
\end{align*}
Given $R>0$, we bound $T_{11}$ as follows:
\begin{align*}
    T_{11}
    \leq&\,
     L (1+2R^p) \,\E\left[  |v^n|^2\right]
   +
    L \, \E \left[| v^n |^2 |V^n|^p \1_{| V^n | > R}\right]
+
    L \, \E \left[| v^n |^2 |V^{n+1}|^p \1_{| V^{n+1} | > R}\right]
    \\
    \leq&\,
    L\, (1+2 R^p) \gamma^n(t)
    +
    L \, \E\left [ | v^n |^4 |V^n|^{2p} \right]^{1/2}
    \,\E \left[\1_{| V^n | > R} \right]^{1/2}\\
&\, +
    L \, \E\left [ | v^n |^4 |V^{n+1}|^{2p} \right]^{1/2}
    \,\E \left[\1_{| V^{n+1} | > R} \right]^{1/2}
    \\
    \leq&\,
    C\, (1+R^p) \gamma^n(t)
    +
    C \,\E \left[\1_{| V^n | > R} \right]^{1/2}+
    C \,\E \left[\1_{| V^{n+1} | > R} \right]^{1/2},
\end{align*}
where we have used the uniform-in-$n$ bound on moments of $f^n$
obtained in \eqref{uniforminnbounds}. For the term $T_{12}$, we
get
\begin{align*}
    T_{12}
    &\leq
     L (1+R^p)\, \E \left[|v^n||x^n|\right]
    + L \, \E \left[ | v^n | \, | x^n | \,   | V^n |^p \1_{| V^n | > R}\right]
    \\
   &\leq
    \frac{L}{2} \, (1+R^p) \gamma^n(t)
    + \frac{L}{2}  \, \E \left[ | x^n |^2\right]
    + \frac{L}{2}  \, \E \left[ | v^n |^2 \,  | V^n |^{2p} \1_{| V^n | > R}\right]
    \\
    &\leq
    L\, (1+R^p) \gamma^n(t)
    + \frac{L}{2}  \,
    \E \left[ | v^n |^4 \,  | V^n |^{4p} \right]^{1/2}
    \E \left[\1_{| V^n | > R} \right]^{1/2}
    \\
   &\leq
    L\, (1+R^p) \gamma^n(t)
    +
    C \, \E \left[\1_{| V^n | > R} \right]^{1/2},
\end{align*}
using again the bound on moments of $f^n$ in
\eqref{uniforminnbounds}. Finally, using
\eqref{uniftailinnbounds}, there exist constants $b$ and $C$ such
that for all $0\leq t \leq T$
  \begin{equation}
    \label{boundT_1}
    T_1 \leq C (1+R^p) \, \gamma^n (t) + C \, e^{- b R^p}
  \end{equation}
for all $n$ and  $R>0$.

{\it Estimate for $T_2$.} On the other hand, for $T_2$,
\begin{align}
    T_2
    & = \!- \!\E \left[v^n \cdot \Big((H*f^n_t)(Z^{n+1}) - (H*f^{n}_t)(Z^{n})\Big)
     \right]\!-\!
    \E \left[ v^n \cdot \Big( H*(f^n_t - f^{n-1}_t)(Z^{n}) \Big)
   \right] \nonumber\\
    &=: T_{21} + T_{22}.    \label{eq:ex10}
\end{align}
For the first term $T_{21}$, we proceed analogously to the
estimates of $T_{11}$ and $T_{12}$ to obtain
\begin{align}
    T_{21} =&\,
    - \E \left[v^n \cdot \int_{\R^{2d}} (H(Z^{n+1}-z) - H(Z^{n}-z))
    f^n_t(x,v) \,dx\,dv\right]
    \nonumber\\
    \leq&\,L\,
    \E \left[|v^n| \int_{\R^{2d}} |Z^{n+1} - Z^{n}|(1 + |V^n|^p + |V^{n+1}|^p + |v|^p)
    f^n_t(x,v) \,dx\,dv\right]
    \nonumber\\
    \leq&\,
    C\,
    \E \left[ |v^n| |Z^{n+1} - Z^{n}| \left(1 + |V^n|^p + |V^{n+1}|^p
    \right)\right] \nonumber\\
    \leq&\, C (1+R^p) \, \gamma^n (t) + C \, e^{- b R^p}
    \label{eq:ex11}
\end{align}
where the last steps where not detailed since they are very similar to
the estimates of $T_{11}$ and $T_{12}$, and the uniform moment bounds
\eqref{uniforminnbounds} and \eqref{eq:H-loc-Lipschitz} were used.
Now, for $T_{22}$, observe that, taking $g^{n} := \law ((X^{n},
V^n,X^{n-1}, V^{n-1}))$, we can write the following identity
$$
A:=H*(f^n_t - f^{n-1}_t)(Z^{n}) \! = \!\int_{\R^{4d}}
\!\!\big(H(X^n - x, V^n - v) - H(X^n - y, V^n - w)\big) \,dg^{n}
$$
where we used the shortcut notation $dg^{n}$ for the measure
$dg^n(x,v,y,w)$. By the Cauchy-Schwarz inequality and the uniform
moment bounds \eqref{uniforminnbounds}, we get
\begin{align*}
    |A| \leq&\,
    L \int_{\R^{4d}} (|x-y|+|v-w|)
    (1 + |V^n|^p + |v|^p + |w|^p) \,dg^{n}(x,v,y,w)
    \\
    \leq&\,
    L (1 + |V^n|^p) \,\E \left[|x^{n-1}| + |v^{n-1}|\right]
    \\
    &\,+
    L \left(
      \int_{\R^{4d}} (|x-y|+|v-w|)^2 \,dg^{n}\right)^{1/2}
    \left(
      \int_{\R^{4d}} (|v|^p + |w|^p)^2 \,dg^{n}
    \right)^{1/2}
    \\
    \leq&\,
    C (1 + |V^n|^p) \E \left[|x^{n-1}| + |v^{n-1}|\right]
    +
    C \left( \E \left[ |Z^{n} - Z^{n-1}|^2 \right] \right) ^{1/2}
    \\
    \leq&\,
    C (1 + |V^n|^p) \, \gamma^{n-1}(t)^{1/2}\, .
\end{align*}
Using the expression of $T_{22}$, we deduce
\begin{align}
    T_{22} \leq &\, \E \left[ |v^n| | H*(f^n_t - f^{n-1}_t)(Z^{n}) |\right]
    \leq
    C \gamma^{n-1}(t)^{1/2} \, \E \left[ |v^n| (1 + |V^n|^p)\right]
    \nonumber\\
    \leq&\,
    C \gamma^{n-1}(t)+
    \E\left[ |v^n|^2\right] \E\left[ |V^n|^{2p} \right]
    \leq
    C \gamma^{n-1}(t)+
    C \gamma^n(t)\,,     \label{eq:ex12}
\end{align}
again by the Cauchy-Schwarz inequality and the uniform moment bounds
\eqref{uniforminnbounds}.

Hence, putting \eqref{boundT_1}, \eqref{eq:ex10}, \eqref{eq:ex11}
and  \eqref{eq:ex12} in \eqref{eq:ex0.5} and \eqref{eq:ex1}, we
conclude
\begin{equation*}
    \frac{d}{dt} \gamma^n(t)
    \leq
    C(1+R^p) \gamma^n(t)
    +
    C \gamma^{n-1}(t)
    +
    C e^{-b R^p}.
\end{equation*}

\bigskip

{\it Induction Argument.} Taking $R > 1$ we may write that
\begin{equation}
    \label{eq:ex14}
    \frac{d}{dt} \gamma^n(t)
    \leq
    C \left( r \gamma^n(t)
    +
    \gamma^{n-1}(t)
    +
    e^{-r}
    \right),
\end{equation}
for some other constant $C > 0$ and for any all $r > 1$.
Gronwall's Lemma then proves that
$$
    \gamma^n(t)
    \leq
    C \int_0^t e^{Cr (t-s)} \gamma^{n-1}(s)\,ds
    + C e^{-r} t e^{Crt},
$$
and iterating this inequality gives
\begin{align*}
      \gamma^n(t)
    &\leq
    C^n \int_0^t e^{Cr (t-s)} \gamma^{0}(s)
    \frac{(t-s)^{n-1}}{(n-1)!} \,ds
    +
    C t e^{-r} e^{Crt} \sum_{i=0}^{n-1} \frac{C^i t^i}{(i+1)!}
    \\
    &\leq
    C^n e^{Crt} \frac{t^{n-1}}{(n-1)!} \int_0^t \gamma^{0}(s)
     \,ds
    +
    C t e^{-r} e^{Crt} e^{Ct}
    \\
    &\leq
    C^n e^{Crt} t^{n} \sup_{s \in [0,t]} \gamma^{0}(s)
    +
    C t e^{Ct} e^{r(Ct-1)}.
\end{align*}
Taking $r = n$, we obtain
$$
    \gamma^n(t)
    \\
    \leq
    \exp\left(
      n(\ln (Ct) + Ct)
    \right)
    \sup_{s \in [0,t]} \gamma^{0}(s)
    +
    C t e^{Ct} e^{n (Ct - 1)}.
$$
Choosing $0<T_* <T $ small enough such that $\lambda := \max\{CT_*
- 1, \ln(CT_*) + CT_*\} < 0$, then
  \begin{equation}
    \label{eq:ex18}
    \sup_{t \in [0,T_*]} \gamma^n(t)
    \\
    \leq e^{\lambda n}
    \left(
      \sup_{s \in [0,T_*]} \gamma^{0}(s)
      +
      C T e^{CT}
    \right).
\end{equation}
Since by definition $W_2(f^{n+1}_t, f^n_t)^2 \leq \E|Z^{n+1}_t -
Z^n_t|^2 = \gamma^n(t)$, we conclude that the sequence of curves $\{t
\in [0,T_*] \mapsto f^n_t\}_{n \geq 0}$ is a Cauchy sequence in the
metric space $C([0,T_*], \mathcal{P}_2(\R^{2d}))$ equipped with the
distance
\begin{equation*}
    \mathcal{W}_2(f,g) := \sup_{t \in [0,T_*]} W_2(f_t,g_t).
\end{equation*}
By completeness of this space, we define $f \in C([0,T_*],
\mathcal{P}_2(\R^{2d}))$ by $f_t := \lim_{n \to +\infty} f^n_t$ for $t
\in [0,T_*]$.

This convergence and the uniform moment bounds on $f^n$ in
\eqref{uniforminnbounds} allow us to pass to the limit
in~\eqref{eq:fn-PDE-weak}. Let us point out how to deal with the
nonlinear term in the equation: observe first that for fixed $s$, it
is given by
\begin{equation}\label{eq:nltermfn}
\int_{\rr^{2d}} \nabla_v \varphi_s \cdot
  H * f^{n-1}_s \, df^n_s
  =
 \int_{\rr^{4d}} \nabla_v \varphi_s (x,v)  \cdot
  H(x-y, v-w)  df^{n-1}_s (y,w) \, df^n_s (x,v).
\end{equation}

But on the one hand $f_s^{n-1}$ and $f_s^{n}$ converge to $f_s$ for
the $W_2$ topology, hence so does $f_s^{n-1} \otimes f_s^{n}$ to $f_s
\otimes f_s$ (in $\rr^{4d}$). On the other hand
$$
\vert \nabla_v \varphi_s (x,v)  \cdot  H(x-y, v-w) \vert
\leq \Vert \nabla_v \varphi_s \Vert_{L^{\infty}} (1 + \vert v \vert + \vert w \vert)
$$
by~\eqref{eq:H-two-sided-sys}. Therefore~\eqref{eq:nltermfn}
converges to
$$
 \int_{\rr^{4d}} \nabla_v \varphi_s (x,v)  \cdot
  H(x-y, v-w)  df_s (y,w) \, df_s (x,v)
=
\int_{\rr^{2d}} \nabla_v \varphi_s \cdot
  H * f_s \, df_s
  $$
  for all $s$. Uniform-in-$s$ bounds finally allow to pass to the limit in the integral in $s$.

  With this, we have shown that $f_t$ is a solution on $[0,T_*]$ of
  the nonlinear PDE \eqref{eq:pde}. Now, one can extend the solution
  to the whole interval $[0,T]$ by iterating this procedure,
  starting from $T_*$. This can be done since the additional time
  $T_*'$ for which we can extend a solution starting at $T_*$ depends
  only on moment bounds on $f_{T_*}$, for which we have the bound
  \eqref{uniforminnbounds}, valid up to $T$.

\paragraph{Step 4.- Existence for the nonlinear SDE:} Now we use $f_t$ to
define the process $(X_t,V_t)$ by
\begin{equation}
    \label{eq:XV-from-f}
    \left\{
    \begin{split}
      &dX_t=V_t\,dt
      \\
      &dV_t =
      \sqrt{2}\, dB_t - F(X_t, V_t)\,dt
      -  (H * f_t) (X_t,V_t)\,dt,
      \\
      &(X_0,V_0) = (X^0, V^0)
    \end{split}
    \right.
\end{equation}
thanks to Lemma~~\ref{lem:edslin}. Observe that for all $t$, $f_t$
is the $W_2$-limit of $f^n_t$ and $p \leq 2$, so that
$$
\int_{\rr^{2d}} \vert v \vert^{p} df_t(x,v)
=
 \int_{\rr^{2d}} \vert v \vert^{p}
  df^n_t(x,v)\leq C,
$$
uniformly in $t \in [0,T]$ since the $f^n_t$ have second moments
bounded according to Lemma~\ref{lem:momscheme}. If $g_t$ is the
law of $(X_t,V_t)$, then, as in section~\ref{sec:linearPDE}, $g_t$
is a weak solution of the linear PDE
  \begin{equation*}
    \partial_t g_t + v \cdot \nabla_x g_t  = \Delta_v g_t + \nabla_v( (F + H*f_t) g_t).
  \end{equation*}
Of course, $f_t$ is also a solution of the same linear PDE; by
uniqueness of solutions to this linear PDE (see again section~\ref{sec:linearPDE}), we deduce that $f_t = g_t$, and hence $(X_t,V_t)$ is a solution to the nonlinear SDE
\eqref{eq:nlSDE} on $[0,T]$.

\paragraph{Step 5.- Uniqueness for the nonlinear PDE~\eqref{eq:pde}:}
Now, take two solutions $f^1$, $f^2$ of the nonlinear PDE, and define
the processes $(X^1_t, V^1_t)$ and $(X^2_t, V^2_t)$ by
\eqref{eq:XV-from-f}, putting $f^1$ and $f^2$ in the place of $f$,
respectively. As the law of $(X^i_t, V^i_t)$ ($i = 1,2$) solves the
linear PDE \eqref{eq:linearPDE} with $f_t^i$ instead of $g_t$, so this
law must in fact be $f_t^i$ by uniqueness of the linear PDE. Then, if
we follow the same calculation we did in step 3 above, we obtain the
following instead of \eqref{eq:ex14}:
\begin{equation}
    \label{eq:u1}
    \frac{d}{dt} \gamma(t)
    \leq
    C r \gamma(t)
    +
    C\, e^{-r},
\end{equation}
for a constant $C$ and any $r \geq 1$, with $\gamma(t) := \E
\left[|X^1_t - X^2_t|^2\right] + \E \left[|V^1_t -
V^2_t|^2\right]$. For the above to be valid, we need a bound on
exponential moments of $f^1_t$ and $f^2_t$. This estimate can be
obtained in a similar way as in Lemma \ref{lem:momscheme} and
therefore, we omit the proof:

\begin{lem}
  \label{lem:mompde}
  Assume hypotheses
  \eqref{eq:F-one-sided-growth}--\eqref{eq:H-loc-Lipschitz} on $F$ and
  $H$. Let $(f_t)_{t \geq0}$ be a solution to \eqref{eq:pde} with
  initial datum a probability measure $f_0$ on $\rr^{2d}$ such that
$$
    \int_{\rr^{2d}} \big( | x |^2 + e^{a |v|^p}\big)  f_0(x,v) \, dx \, dv < + \infty
$$
for a positive $a$. Then for all $T$ there exists $b > 0$ which
depends only on $f_0$ and $T$, such that
$$
    \sup_{0 \leq t \leq T}
    \int_{\rr^{2d}} \big( | x |^2 + e^{b |v|^p} \big)  f_t(x,v) \, dx \, dv
    < + \infty.
$$
\end{lem}

\

\noindent Observe now that
$$
 \gamma(0)  = \E \left [|X^1_0 - X^2_0|^2 + |V^1_0 - V^2_0|^2\right] = 0
$$
since $X^1_0 = X^2_0 = X^0$, and similarly for $V$.

Assume now that $\gamma$ is non identically $0$. Then, with the same
argument as in step 3 of the proof of Theorem \ref{thm:main}, whenever
$0 < \gamma (t) < 1/e$ we can choose $r := -\ln \gamma(t)$ in
\eqref{eq:u1} to obtain
\begin{equation}
  \label{eq:difineq}
    \frac{d}{dt} \gamma(t)
    \leq
    - C \gamma(t) \ln \gamma(t)
    +
     C \gamma(t)
 \leq -2 C \gamma(t) \ln \gamma(t).
\end{equation}
If $\gamma(t) = 0$ at some point, then one can see that $\frac{d}{dt}
\gamma(t) \leq 0$ by letting $r \to +\infty$ in \eqref{eq:u1}, and in
that case the inequality \eqref{eq:difineq} holds trivially (again
setting $z \ln z = 0$ at $z=0$ by continuity). Hence,
\eqref{eq:difineq} holds as long as $0 \leq \gamma(t) < 1/e$. By
Gronwall's Lemma, this implies that $\gamma(t) = 0$ for $t \in [0,T]$
and shows that $f^1_t$ and $f^2_t$ coincide, proving that solutions to
the nonlinear PDE are unique.

\paragraph{Step 6.- Uniqueness for the nonlinear SDE~\eqref{eq:nlSDE}:}
Take two pairs of stochastic processes $(X^1_t, V^1_t)$ and $(X^2_t,
V^2_t)$ which are solutions to the nonlinear SDE
\eqref{eq:nlSDE}. Then, their laws are solutions to the nonlinear PDE
\eqref{eq:pde}, and by the previous step we know that they must be the
same. If we call $f_t$ their common law, then both $(X^1_t, V^1_t)$
and $(X^2_t, V^2_t)$ are solutions to the linear SDE \eqref{eq:edslin}
with $f_t$ instead of $g_t$, and by uniqueness of this linear SDE (see
section~\ref{sec:linSDE}), they must coincide.

\bigskip

\noindent This concludes the proof of
Theorem~\ref{thm:mainexist}.

\bigskip
{\footnotesize \noindent\textit{Acknowledgments.} The last two
authors acknowledge support from the project MTM2008-06349-C03-03
DGI-MCI (Spain) and 2009-SGR-345 from AGAUR-Generalitat de
Catalunya. All authors were partially supported by the
ANR-08-BLAN-0333-01 Projet CBDif-Fr. This work was initiated while the first author was visiting UAB; it is a pleasure for him to thank this institution for its kind hospitality.
}

\end{document}